\documentclass{amsproc}
\usepackage{amsmath}
\usepackage{amssymb}
\usepackage{amsfonts}

\theoremstyle{plain}

\newtheorem{corollary}{Corollary}

\newtheorem{definition}{Definition}

\newtheorem{lemma}{Lemma}
\newtheorem{notation}{Notation}

\newtheorem{proposition}{Proposition}
\newtheorem{remark}{Remark}

\newtheorem{theorem}{Theorem}
\numberwithin{equation}{section}
\input{tcilatex}

\begin{document}
\title[Some Remarks on Spectrum of Nonlinear Continuous Operators]{Some
Remarks on Spectrum of Nonlinear Continuous Operators}
\author{Kamal N. Soltanov}
\address{{\small National Academy of Sciences of Azerbaijan, Baku, AZERBAIJAN%
}}
\email{sultan\_kamal@hotmail.com}
\urladdr{https://www.researchgate.net/profile/Kamal-Soltanov/research}
\date{}

\begin{abstract}
In this article, the existence of the spectrum (the eigenvalues) for the
nonlinear continuous operators acting in the Banach spaces is investigated.
For the study, this question is used a different approach that allows the
studying of all eigenvalues of the nonlinear operator relative to another
nonlinear operator. Here shows that in nonlinear operators case is necessary
to seek the spectrum of the given nonlinear operator relative to another
nonlinear operator satisfying certain conditions. The different examples,
for which eigenvalues can find are provided. Moreover, the nonlinear
problems including parameters are studied.
\end{abstract}

\maketitle

\section{\protect\bigskip Introduction}

Well-known that the spectral theory for linear operators is one of the most
important topics of linear functional analysis, as in many cases for the
study of the linear operator, it is enough to study its spectrum. It should
be noted the spectral theory of linear operators has essential application
in the many topics of the natural sciences (moreover, the spectral theory is
one of the areas that play a fundamental role in quantum mechanics). As, the
many processes of physics, mechanics, biology, etc. at the mathematical
describing, usually would be nonlinear problems, consequently, the operators
generated by these are nonlinear operators acting in Banach spaces.

It needs to be noted in the literature there exist sufficiently many works,
to the finding the first eigenvalue of the nonlinear continuous operators
devoted. In these works, various definitions of the spectra for the specific
classes of nonlinear continuous operators were introduced. In many works,
the nonlinear equations of type $f\left( x\right) -\lambda g\left( x\right)
=0$ in the appropriate spaces were considered, where $f$ is the basic
operator, $g$ is the compact operator, and an infimum of the parameter $%
\lambda $ was found, at which, in the main, the existence of a solution of
the considered equation studied (see, e.g. \cite{3, 5, 6, 9, 10, 11, 24, 36}%
, etc.). In works \cite{6, 7, 8, 9, 10, 11, 16, 23, 26} the Strum-Liouville
type problem for the perturbed by nonlinear operators of linear operators
was investigated. More exactly, the operator in form $F\left( \lambda
,x\right) =\lambda Lx+g\left( \lambda ,x\right) $ was considered, and the
bifurcation of solutions to the examined problems was studied.

Well-known that the founding of eigenvalues of linear operators allows for
studying the bifurcations of solutions, which appear under the investigation
of the semilinear equations. It should be noted that there are works, where
the existence of the first eigenvalue for certain nonlinear smooth operators
was studied, and using it the bifurcations of solutions to nonlinear
equations with such operators were investigated (see, e.g., \cite{12, 16, 33}%
, etc. \footnote{%
see, also Lopez-Gomez, J. \textit{Spectral Theory and Nonlinear Functional
Analysis} (1st ed.). (2001) Chapman and Hall/CRC, \textit{T\&F Books}
https://doi.org/10.1201/9781420035506\textit{\ }
\par
{}}).

In works \cite{1, 2, 4, 13, 14, 15, 18, 19, 20, 21, 22, 25, 27, 28} have
been introduced the definition of the spectra (more exactly, first
eigenvalue) using the equation in the form $f\left( x\right) -\lambda
_{1}Ix=0$, as in the linear operators theory. And also have been introduced
the definition of the spectra for the classes operators for the Frechet
differentiable operators, operators that satisfy the Lipschitz condition,
operators that are a special class of continuous operators, and operators
that are linearly bounded were defined. These approaches supposed that the
spectrum of the operator $f$ acting in the Banach space $X$ can be defined,
as in the theory of linear operators. These approaches and obtained results
in enough form were explained in the book \cite{1} (see, also the survey 
\cite{2}). In the above works, the study used degree theory that requires
the condition compactness, this condition in what follows was generalized
and used the Kuratowski measure of noncompactness. The definitions of the
spectrum introduced in these works, couldn't satisfy the next requirements
since a found parameter $\lambda $ will be a function of elements of the
domain(see, provided examples below).

It will be best if one can introduce such a definition of the spectrum of
the continuous nonlinear operator that satisfies some basic requirements,
which were analogous to properties of the spectrum the existing in the
linear operator's theory then one could be to seek and also other spectrums
(i.e. eigenvalues and eigenvectors). So, from the explanation below will be
seen that in order for the spectrum can to characterize the nonlinear
operator, needs to approach another way to the definition of the spectrum of
the nonlinear operator.

This paper is proposed a new approach for the study of the spectrum of
continuous nonlinear operators in the Banach spaces. In reality, here we
find the first eigenvalue of the nonlinear continuous operator in Banach
space, in addition, this approach shows how one can seek the other
eigenvalues. Here shows that if to use the proposed definition of the
spectrum of nonlinear continuous operators in Banach spaces, then the
spectra will satisfy the certain properties that are similar to properties,
having in the linear operator theory. Moreover, here we investigate also the
solvability of nonlinear equations in the Banach spaces.

In this paper, we will study the spectrum of nonlinear operators acting in
Banach spaces, and also the solvability of the nonlinear equations,
dependent on parameters, using the general solvability theorems and
fixed-point theorems of the works \cite{29, 30, 31, 32}\footnote{%
see, also Kamal N. Soltanov, \ Nonlinear Operators, Fixed-Point Theorems,
Nonlinear Equations, Springer, V.V. Mityushev, M.V. Ruzhansky (eds.),
Current Trends in Analysis and Its Applications, Trends in Mathematics
(2015), 347-354, DOI 10.1007/978-3-319-12577-0\_41}.

Let $X,Y$ be real Banach spaces on the field $\Re $ and $X^{\ast },Y^{\ast }$
be of their dual spaces, let $Y$ be reflexive space. Let $f:X$ $%
\longrightarrow Y$ and $g:X$ $\longrightarrow Y$ be nonlinear continuous
operators such, that $f\left( 0\right) =0$, $g\left( 0\right) =0$. For
investigation of the spectrum of the continuous nonlinear operators, we will
consider the following equation 
\begin{equation}
f\left( x\right) =\lambda g\left( x\right) ,\quad x\in M\subseteq X,
\label{P}
\end{equation}%
where $f,\ g$ are continuous operators, in particular, $g$ can be the
identical operator, in addition, we study the solvability of the equation,
dependent on a parameter $\lambda $

\begin{equation}
f_{\lambda }\left( x\right) \equiv f\left( x\right) -\lambda g\left(
x\right) =y,\quad for\ y\in Y,  \label{Pa}
\end{equation}%
where $\lambda $ is an element of $%
\mathbb{C}
$, generally.

The paper is organized as follows. In section 2, the definition for the
spectrum of the nonlinear continuous operators in the Banach spaces is
provided, and also some complementaries to the definition, and examples are
provided. Moreover, the general solvability theorems and the fixed-point
theorems are proved. Here how can be to find the first eigenvalue of
nonlinear continuous operators relative to another nonlinear continuous
operator is shown. Section 3 some examples of nonlinear differential
operators, for which found first eigenvalues relative to other nonlinear
differential operators provided, and the relations between found first
eigenvalues with first eigenvalues of linear differential operators showed.
Section 4 studied the existence of the first eigenvalues of the fully
nonlinear continuous operator relative to other nonlinear continuous
operators and provided examples.

\section{Spectral properties of nonlinear continuous operators}

We will consider the nonlinear continuous operators, acting in the Banach
spaces, and will introduce the concept for the spectrum of the nonlinear
continuous operator relative to another nonlinear continuous operator.

Let$\ X$ and $Y$ be the real Banach spaces, $F:D\left( f\right)
=X\longrightarrow Y$, $G:X\subseteq D\left( G\right) \longrightarrow Y$ be
nonlinear bounded continuous operators (for generality) and $\lambda \in 
\mathbb{C}
$ be the number, and $F\left( 0\right) =0$, $G\left( 0\right) =0$.

So, we will investigate the spectrum of operator $F$ relative to operator $G$%
, for this, in the beginning, we will study the solvability of the following
equation with a parameter $\lambda $ \ 
\begin{equation}
f_{\lambda }\left( x\right) \equiv F\left( x\right) -\lambda G\left(
x\right) =0,\ \text{or }F\left( x\right) =\lambda G\left( x\right) ,\ x\in X.
\label{2.1}
\end{equation}%
And also we will study the following equation 
\begin{equation}
f_{\lambda }\left( x\right) \equiv F\left( x\right) -\lambda G\left(
x\right) =y,\quad y\in Y.  \label{2.1a}
\end{equation}

We will introduce concepts that be necessary for this paper.

\begin{definition}
\label{D_1}The operator $f:D\left( f\right) \subseteq X\longrightarrow Y$ is
called bounded if there is a continuous function $\mu
:R_{+}^{1}\longrightarrow R_{+}^{1}$ such that 
\begin{equation*}
\left\Vert f\left( x\right) \right\Vert _{Y}\leq \mu \left( \left\Vert
x\right\Vert _{X}\right) ,\quad \forall x\in D\left( f\right) ,
\end{equation*}%
and denote the class of such operators as $\mathfrak{B}$ and the bounded
continuous class of operators by $\mathfrak{B}C^{0}$.
\end{definition}

Introduce the order of relationships in the class of the bounded continuous
operators acting in the Banach spaces.

\begin{definition}
\label{D_2} Let $X_{0},Y_{0}$ be Banach spaces, and $F:D\left( f\right)
\subseteq X_{0}$ $\longrightarrow Y_{0}$ , $G:D\left( G\right) \subseteq
X_{0}\longrightarrow Y_{0}$ be continuous operators. Denote by $\mathcal{F}%
_{F}\left( Z\right) $, $\mathcal{F}_{G}\left( Z\right) $ sets defined in the
form 
\begin{equation*}
\mathcal{F}_{F}\left( Z\right) \equiv \left\{ x\in X_{0}\left\vert
~\left\Vert F\left( x\right) \right\Vert _{Z}\right. <\infty \right\} \neq
\varnothing \ \ \&
\end{equation*}%
\begin{equation*}
\mathcal{F}_{G}\left( Z\right) \equiv \left\{ x\in X_{0}\left\vert
~\left\Vert G\left( x\right) \right\Vert _{Z}\right. <\infty \right\} \neq
\varnothing ,
\end{equation*}%
which are subsets of $X_{0}$ for each Banach space $Z\subseteq Y_{0}$\
satisfying conditions $\Re \left( F\right) \cap Z\neq \varnothing $, $\Re
\left( G\right) \cap Z\neq \varnothing $. If the inclusion $\mathcal{F}%
_{F}\left( Z\right) \subset \mathcal{F}_{G}\left( Z\right) $ holds for each
Banach space $Z\subseteq Y_{0}$, then we will say that operator $F$ is
greater than operator $G$, and denote it as $F\succ G$.
\end{definition}

\begin{definition}
\label{D_S} Let $F:X\longrightarrow Y$ and $G:X\subseteq D\left( G\right)
\longrightarrow Y$, moreover, $F\succ G$. An element $\lambda \in 
\mathbb{C}
$ is called a regular and belongs to the $G$-resolvent's of the operator $F$
iff $f_{\lambda }^{-1}\equiv \left( F-\lambda G\ \right) ^{-1}:F\left(
X\right) \cap G\left( X\right) \subseteq Y\longrightarrow X$ exists and $%
f_{\lambda }^{-1}\equiv \left( F-\lambda G\ \right) ^{-1}\in \mathfrak{B}%
C^{0}$, this subset of $%
\mathbb{C}
$ denoted by $\rho _{G}\left( F\right) \subseteq 
\mathbb{C}
$, (i.e. $\lambda \in \rho _{G}\left( F\right) $), where $f_{\lambda }\left(
\cdot \right) \equiv F\left( \cdot \right) -\lambda G\left( \cdot \right) $.

Consequently, an element $\lambda \in 
\mathbb{C}
$ is called a spectrum iff it belongs to the $%
\mathbb{C}
-$ $\rho _{G}\left( F\right) $, which is defined as the $G-$spectrum of the
operator $F$, and denoted by $\sigma _{G}\left( F\right) $ (i.e. $\lambda
\in \sigma _{G}\left( F\right) \equiv 
\mathbb{C}
-$ $\rho _{G}\left( F\right) $).
\end{definition}

\begin{notation}
\label{N_2}The above definition of the spectrum does not suitable for all
pairs of operators, which are chosen in the independent way, which will be
shown below. We will call $\lambda $ the first eigenvalue of the examined
operator relative to another operator as in Definition \ref{D_S}, which is
independent of elements from the domain of the examined operators.

Then the definition will allow seeking, in the above sense, also the
following eigenvalues of the examined operator relative to another operator.
\end{notation}

So, for simplicity, we start to consider the case, when $F\succ G$ and when
one of these operators has the inverse operator from the class $\mathfrak{B}%
C^{0}$. If we assume that operator $F$ is invertible, as operator $%
F^{-1}:F\left( X\right) \subseteq Y\longrightarrow X$ then, using $F^{-1},$
we get the equation 
\begin{equation}
y-\lambda G\left( F^{-1}\left( y\right) \right) =0,\quad y=F\left( x\right)
,\ x\in X,  \label{2.2}
\end{equation}%
that needs to study on the subset $F\left( X\right) \subseteq Y$. Thus, we
will derive an equation that is equivalent to the examined equation, for
which the existence of the first eigenvalue in many works was investigated,
in their sense (see, \cite{1, 2, 3, 4, 6, 9, 13, 18, 19, 20, 22, 23, 25, 28}
and their references). Unlike a usual case, here the operator $G\circ F^{-1}$
is defined on the subset $F\left( X\right) $ and acts as $G\circ
F^{-1}:F\left( X\right) \longrightarrow G\left( X\right) \subseteq Y$. If
the operator $G$ is invertible then in the same way as above, we get to the
equation 
\begin{equation*}
F\left( G^{-1}\left( y\right) \right) -\lambda y=0,\quad y=G\left( x\right)
,\ x\in X,
\end{equation*}%
where $G^{-1}$ is the inverse operator to $G$. Consequently, in this case,
the obtained equation will need to investigate on the subset $G\left(
X\right) $ of $Y$.

Thus, if assume the operator $F$ (or $G$) is invertible then, we obtain the
equation

\begin{equation}
\widetilde{f}_{\lambda }\left( y\right) \equiv \lambda ^{-1}Iy-G\left(
F^{-1}\left( y\right) \right) =0,\quad \widetilde{f}_{\lambda }:D\left( 
\widetilde{f}_{\lambda }\right) \subseteq Y\longrightarrow Y  \label{2.3}
\end{equation}%
consequently, the finding of a first eigenvalue of the operator $F$ relative
to operator $G$ is transformed into the finding of a first eigenvalue of the
operator $G\circ F^{-1}$ (or $F\circ G^{-1}$).

It is clear if assume the operator $F$ is the linear continuous operator
having the inverse operator $F^{-1}$, then the equation (\ref{2.2}) is
equivalent to the equation 
\begin{equation}
\lambda ^{-1}x-F^{-1}\circ G\left( x\right) =0,  \label{2.4}
\end{equation}%
consequently, the finding of a first eigenvalue of the operator $F$ relative
to operator $G$ is transformed into the finding of a first eigenvalue of the
operator $F^{-1}\circ G$ . Problems of such types were studied in many
articles (\cite{9, 10, 11, 17, 26, 35}, etc.). We would like to study the
problem ((\ref{2.1}) in the general case. Section 3 will be given
explanations relative to the previous cases.

Before starting the investigation of the spectrum of the nonlinear operator
relative to other nonlinear operators in the general case, it is necessary
to investigate the solvability of the nonlinear equation (\ref{2.1a}). We
will use the general existence and fixed-point theorems of articles \cite%
{29, 30} to investigate the main equations. In the beginning, we will lead
the mentioned results from these articles.

\subsection{General Solvability Results}

Let $X,Y$ be real Banach spaces such as above, $f:D\left( f\right) \subseteq
X\longrightarrow Y$ be an operator, and $B_{r_{0}}^{X}\left( 0\right) $ $%
\subseteq D\left( f\right) $ is the closed ball with a center of $0\in X$.

Consider the following conditions.

(\textit{i}) $f:D\left( f\right) \subseteq X\longrightarrow Y$ is the
nonlinear bounded continuous operator;

(\textit{ii}) There is a such mapping $g:X\subseteq D\left( g\right)
\longrightarrow Y^{\ast }$ that$\ g\left( B_{r_{0}}^{X}\left( 0\right)
\right) =B_{r_{1}}^{Y^{\ast }}\left( 0\right) $ and 
\begin{equation*}
\left\langle f\left( x\right) ,\widetilde{g}\left( x\right) \right\rangle
\geq \nu \left( \left\Vert x\right\Vert _{X}\right) =\nu \left( r\right)
,\quad \forall x\in S_{r}^{X}\left( 0\right)
\end{equation*}%
holds \footnote{%
In particular, the mapping $g$ can be a linear bounded operator as $g\equiv
L:X\longrightarrow Y^{\ast }$ that satisfy the conditions of \textit{(ii). }}%
, where $\widetilde{g}\left( x\right) \equiv \frac{g\left( x\right) }{%
\left\Vert g\left( x\right) \right\Vert }$, $\nu :R_{+}^{1}\longrightarrow
R^{1}$ is the continuous function ($\nu \in C^{0}$), moreover, $\nu \left(
\tau \right) $ is the nondecreasing function for $\tau :$ $\tau _{0}\leq
\tau \leq r_{0}$, and $\nu \left( r_{0}\right) \geq \delta _{0}$; $\delta
_{0}>0$, $\tau _{0}\geq 0$ are constants.

(\textit{iii}) Almost each $x_{0}\in int\ B_{r_{0}}^{X}\left( 0\right) $
possesses such neighborhood $V_{\varepsilon }\left( x_{0}\right) $, $%
\varepsilon \geq \varepsilon _{0}>0$ that the following inequality 
\begin{equation}
\left\Vert f\left( x_{2}\right) -f\left( x_{1}\right) \right\Vert _{Y}\geq
\Phi \left( \left\Vert x_{2}-x_{1}\right\Vert _{X},x_{0},\varepsilon \right)
\label{2.6}
\end{equation}%
holds for any $\forall x_{1},x_{2}\in V_{\varepsilon }\left( x_{0}\right)
\cap B_{r_{0}}^{X}\left( 0\right) $, where $\varepsilon _{0}>0$\ and $\Phi
\left( \tau ,x_{0},\varepsilon \right) \geq 0$ is the continuous function of 
$\tau $ and $\Phi \left( \tau ,\widetilde{x},\varepsilon \right)
=0\Leftrightarrow \tau =0$ (in particular, maybe $x_{0}=0$, $\varepsilon
=\varepsilon _{0}=r_{0}$ and $V_{\varepsilon }\left( x_{0}\right)
=V_{r_{0}}\left( 0\right) \equiv B_{r_{0}}^{X}\left( 0\right) $, consequently%
$\ \Phi \left( \tau ,x_{0},\varepsilon \right) \equiv \Phi \left( \tau
,x_{0},r_{0}\right) $ on $B_{r_{0}}^{X}\left( 0\right) $).

\begin{theorem}
\label{Th_1} Let $X,Y$ be real Banach spaces such as above, $f:D\left(
f\right) \subseteq X\longrightarrow Y$ be an operator, and $%
B_{r_{0}}^{X}\left( 0\right) $ $\subseteq D\left( f\right) $ is the closed
ball centered $0\in D\left( f\right) $. Assume conditions (i) and (ii) are
fulfilled. Then the image $f\left( B_{r_{0}}^{X}\left( 0\right) \right) $ of
the ball $B_{r_{0}}^{X}\left( 0\right) $ is contained in an absorbing subset 
$Y$ and contains an everywhere dense subset of $M$, which is defined as
follows 
\begin{equation*}
M\equiv \left\{ y\in Y\left\vert \ \left\langle y,\widetilde{g}\left(
x\right) \right\rangle \leq \left\langle f\left( x\right) ,\widetilde{g}%
\left( x\right) \right\rangle ,\right. \forall x\in S_{r_{0}}^{X}\left(
0\right) \right\} .
\end{equation*}

Furthermore, if the condition (iii)\ also is fulfilled then the image $%
f\left( B_{r_{0}}^{X}\left( 0\right) \right) $ of the ball $%
B_{r_{0}}^{X}\left( 0\right) $ is a bodily subset of $Y$, moreover $%
B_{\delta _{0}}^{Y}\left( 0\right) \subseteq M$.
\end{theorem}

The proof of this theorem, and also its generalization was provided in \cite%
{29} (see also, \cite{30, 31, 32}]).\footnote{%
We note Theorem 2 is the generalization of Theorem of such type from
\par
Soltanov K.N., On equations with continuous mappings in Banach spaces.
Funct. Anal. Appl. (1999) 33, 1, 76-81.}

The condition (\textit{iii}) can be generalized, for example, as in the
following proposition.

\begin{corollary}
\label{C_1} Let all conditions of Theorem \ref{Th_1} be fulfilled except for
the inequality (\ref{2.6}) of condition (iii) instead that the following
inequality 
\begin{equation}
\left\Vert f\left( x_{2}\right) -f\left( x_{1}\right) \right\Vert _{Y}\geq
\Phi \left( \left\Vert x_{2}-x_{1}\right\Vert _{X},x_{0},\varepsilon \right)
+\psi \left( \left\Vert x_{1}-x_{2}\right\Vert _{Z},x_{0},\varepsilon \right)
\label{2.7}
\end{equation}%
holds, where $Z$ is Banach space and $X\subset Z$ is compact, $\psi \left(
\cdot ,x_{0},\varepsilon \right) :R_{+}^{1}\longrightarrow R^{1}$ is a
continuous function relatively $\tau \in R_{+}^{1}$ and $\psi \left(
0,x_{0},\varepsilon \right) =0$.

Then the statement of Theorem \ref{Th_1} is true.
\end{corollary}

From Theorem \ref{Th_1} immediately follows

\begin{theorem}
\label{Th_2}(\textbf{Fixed-Point Theorem}). Let $X$ be a real reflexive
separable Banach space and $f_{1}:D\left( f_{1}\right) \subseteq
X\longrightarrow X$ be a bounded continuous operator. Moreover, let on
closed ball $B_{r_{0}}^{X}\left( 0\right) \subseteq D\left( f_{1}\right) $
centered of $0\in D\left( f_{1}\right) $, operators $f_{1}$ and $f\equiv
Id-f_{1}$ satisfy the following conditions:

(I) The following inequations 
\begin{equation*}
\left\Vert f_{1}\left( x\right) \right\Vert _{X}\leq \mu \left( \left\Vert
x\right\Vert _{X}\right) ,\quad \ \forall x\in B_{r_{0}}^{X}\left( 0\right) ,
\end{equation*}%
\begin{equation}
\left\langle f\left( x\right) ,\widetilde{g}\left( x\right) \right\rangle
\geq \nu \left( \left\Vert x\right\Vert _{X}\right) ,\quad \forall x\in
B_{r_{0}}^{X}\left( 0\right) ,  \label{2.8}
\end{equation}%
hold, where $f_{1}\left( B_{r_{0}}^{X}\left( 0\right) \right) \subseteq
B_{r_{0}}^{X}\left( 0\right) $, $g:D\left( g\right) \subseteq
X\longrightarrow X^{\ast }$, $D\left( f_{1}\right) \subseteq D\left(
g\right) $ and satisfy the condition (ii) (in particular, $g\equiv
J:X\rightleftarrows X^{\ast }$, i.e. $g$ be a duality mapping), $\mu $ and $%
\nu $ are such functions as in Theorem \ref{Th_1};

(II) Almost each $x_{0}\in intB_{r_{0}}^{X}\left( 0\right) $ possesses such
neighborhood $V_{\varepsilon }\left( x_{0}\right) $, $\varepsilon \geq
\varepsilon _{0}>0$ that the following inequality%
\begin{equation*}
\left\Vert f\left( x_{2}\right) -f\left( x_{1}\right) \right\Vert _{X}\geq
\varphi \left( \left\Vert x_{2}-x_{1}\right\Vert _{X},x_{0},\varepsilon
\right) ,
\end{equation*}%
holds for any $\forall x_{1},x_{2}\in V_{\varepsilon }\left( x_{0}\right)
\cap B_{r_{0}}^{X}\left( 0\right) $, where the function $\varphi \left( \tau
,x_{0},\varepsilon \right) $ the similar conditions such as functions from
the right part (\ref{2.7}) satisfies.

Then operator $f_{1}$ possesses a fixed-point in the closed ball $%
B_{r_{0}}^{X}\left( 0\right) $.
\end{theorem}

Now we introduce the following concept.

\begin{definition}
\label{D_3} An operator $f:D\left( f\right) \subseteq X\longrightarrow Y$
possesses the \textrm{P-}\textit{property} iff each precompact subset $%
M\subseteq \func{Im}f$ of $Y$ contains such subsequence (maybe generalized) $%
M_{0}\subseteq M$ that $f^{-1}\left( M_{0}\right) \subseteq G$ and $%
M_{0}\subseteq f\left( G\cap D\left( f\right) \right) $, where $G$ is a
precompact subset of $X$.
\end{definition}

\begin{notation}
\label{N_1} It is easy to see that the condition (iii) of Theorem \ref{Th_1}
one can replace by the condition: $f$ possesses \textrm{P-}\textit{property}.

It should be noted if $f^{-1}$ is the lower or upper semi-continuous mapping
then operator $f:D\left( f\right) \subseteq X\longrightarrow Y$ possesses of
the \textrm{P-}\textit{property}.
\end{notation}

In the above results, condition \textit{(iii)} is required for the
completeness of the image of considered operator $f$. One can bring also
other sufficient conditions on $f$, at which $\Re \left( f\right) $ will be
the closed subset (see, e.g. \cite{30, 31,32}). In particular, the following
results are true.

\begin{lemma}
\label{L_1} Let $X,Y$ be Banach spaces such as above, $f:D\left( f\right)
\subseteq X\longrightarrow Y$ be a bounded continuous operator, and $D\left(
f\right) $ is a weakly closed subset of the reflexive space $X$. Let $f$ has
a weakly closed graph and for each bounded subset $M\subset Y$ the subset $f$
$^{-1}\left( M\right) $ is the bounded subset of $X$. Then $f$ is a weakly
closed operator.
\end{lemma}

We want to note the graph of operator $f$ is weakly closed iff from $x_{m}%
\overset{X}{\rightharpoonup }x_{0}\in D\left( f\right) $ and $f\left(
x_{m}\right) \overset{Y}{\rightharpoonup }y_{0}\in Y$ follows equality $%
f\left( x_{0}\right) \equiv y_{0}\in \Re \left( f\right) \subset Y$ (for the
general case see \cite{30,31}).

For the proof is enough to note, if $\left\{ y_{m}\right\} _{m=1}^{\infty
}\subset \Re \left( f\right) \subset Y$ is the weakly convergent sequence of 
$Y$ then $f^{-1}\left( \left\{ y_{m}\right\} _{m=1}^{\infty }\right) $ is a
bounded subset of $X$ consequently this has such subsequence $\left\{
x_{m}\right\} _{m=1}^{\infty }$ that $x_{m}\in f^{-1}\left( y_{m}\right) $
and $x_{m}\overset{X}{\rightharpoonup }x_{0}\in D\left( f\right) $ for some
element $x_{0}\in D\left( f\right) $ by virtue of the reflexivity of $X$.

\begin{lemma}
\label{L_2} Let $X,Y$ be reflexive Banach spaces, and $f:D\left( f\right)
\subseteq X\longrightarrow Y$ be a bounded continuous mapping that satisfies
the condition: if $G\subseteq D\left( f\right) $ is a closed convex subset
of $X$ then $f\left( G\right) $ is the weakly closed subset of $Y$. Then if $%
G\subseteq D\left( f\right) $ is a bounded closed convex subset of $X$ then $%
f\left( G\right) $ is a closed subset of $Y$.
\end{lemma}

For the proof enough to use the reflexivity of the space $X$ and properties
of the bounded closed convex subset of $X$ (see, e. g. \cite{29, 30}).

\begin{lemma}
\label{L_3} Let $X$ be a Banach space such as above, $f:X\longrightarrow
X^{\ast }$ be a monotone operator satisfying conditions of Theorem \ref{Th_1}%
, and $r\geq \tau _{1}$ be some number. Then $f\left( G\right) $ is a
bounded closed subset containing a ball $B_{r_{1}}^{X^{\ast }}\left( f\left(
0\right) \right) $ for every such bounded closed convex body $G\subset X$
that $B_{r}^{X}\left( 0\right) \subset G$, where $r_{1}=r_{1}\left( r\right)
\geq \delta _{1}>0$.
\end{lemma}

\subsection{Investigation of equations (\protect\ref{2.1}), (\protect\ref%
{2.1a}) and existence of spectra}

We start with the study of the equation (\ref{2.1a}), in order to understand
the role of the parameter $\lambda $. Let $X,Y$ be real reflexive Banach
spaces, $F:X$ $\longrightarrow Y$ , $G:X\subseteq D\left( G\right) $ $%
\longrightarrow Y$ are nonlinear operators and $B_{r_{0}}^{X}\left( 0\right) 
$ ($r_{0}>0$) be a closed ball centered $0\in X$ that belongs to $D\left(
F\right) $. Since in this work, we will consider only operators acting in
real spaces, therefore will seek real numbers $\lambda _{0}$, under which
the considered equation may be solvable.

Assume on the ball $B_{r_{0}}^{X}\left( 0\right) $ are fulfilled the
following conditions:

1) $F:B_{r_{0}}^{X}\left( 0\right) \longrightarrow Y$ , $G:B_{r_{0}}^{X}%
\left( 0\right) $ $\longrightarrow Y$ are a bounded continuous operators,
i.e. there exist such continuous functions $\mu _{j}:\Re _{+}\longrightarrow 
$ $\Re _{+}$, $j=1,2$ that 
\begin{equation*}
\left\Vert F\left( x\right) \right\Vert _{Y}\leq \mu _{1}\left( \left\Vert
x\right\Vert _{X}\right) ;\quad \left\Vert G\left( x\right) \right\Vert
_{Y}\leq \mu _{2}\left( \left\Vert x\right\Vert _{X}\right) ,
\end{equation*}%
hold for any $x\in B_{r_{0}}^{X}\left( 0\right) $, in addition $F\succ G$ ;

2) Let $f_{\lambda }\equiv F-\lambda G$ is the operator from (\ref{2.1a}).
Assume there exists such parameter $\lambda _{0}\in 
\mathbb{R}
_{+}$ that for each $\left( y^{\ast },r,\lambda \right) $\ exists such $x\in
S_{r}^{X}\left( 0\right) $ that the following inequality 
\begin{equation*}
\left\langle f_{\lambda }\left( x\right) ,y^{\ast }\right\rangle \geq \nu
_{\lambda }\left( \left\Vert x\right\Vert _{X}\right) ,\quad \exists x\in
S_{r}^{X}\left( 0\right) ,\quad g\left( x\right) =y^{\ast }
\end{equation*}%
holds,\ where $\left( y^{\ast },r,\left\vert \lambda \right\vert \right) \in
S_{1}^{Y^{\ast }}\left( 0\right) \times \left( 0,r_{0}\right] \times $ $%
\left( 0,\lambda _{0}\right] $ and $\nu _{\lambda }:\Re _{+}\longrightarrow
\Re $ is the continuous function satisfying condition $\left( ii\right) $ of
Theorem \ref{Th_1}, in this case, $\delta _{0}=\delta _{0\lambda }\searrow 0$
if $\left\vert \lambda \right\vert \nearrow \left\vert \lambda
_{0}\right\vert $.

3) Assume for almost every point $x_{0}$ from $B_{r_{0}}^{X}\left( 0\right) $
there exist numbers $\varepsilon \geq \varepsilon _{0}>0$ and such
continuous on $\tau $ functions $\varphi _{\lambda }\left( \tau
,x_{0},\varepsilon \right) \geq 0$, $\psi _{\lambda }(\tau
,x_{0},\varepsilon )$ that the following inequality 
\begin{equation*}
\left\Vert f_{\lambda }\left( x_{1}\right) -f_{\lambda }\left( x_{2}\right)
\right\Vert _{Y}\geq \varphi _{\lambda }\left( \left\Vert
x_{1}-x_{2}\right\Vert _{X},x_{0},\varepsilon \right) +\psi _{\lambda
}(\left\Vert x_{1}-x_{2}\right\Vert _{Z},x_{0},\varepsilon )
\end{equation*}%
holds for any $x_{1},x_{2}\in B_{\varepsilon }^{X}\left( x_{0}\right) $,
where $\varphi \left( \tau ,x_{0},\varepsilon \right) =0\Leftrightarrow \tau
=0$, $\psi _{\lambda }(\cdot ,x_{0},\varepsilon ):\Re _{+}\longrightarrow $ $%
\Re $, $\psi _{\lambda }(0,x_{0},\varepsilon )=0$ for any $\left(
x_{0},\varepsilon \right) $ and $Z$ is the Banach space such that $X\subset
Z $ is compact.

\begin{theorem}
\label{Th_3} Let conditions 1, 2, and 3 be fulfilled on the closed ball $%
B_{r_{0}}^{X}\left( 0\right) \subset X$. Then equation (\ref{2.1a}) is
solvable for $\forall \widetilde{y}\in V_{\lambda }\subset Y$ and each $%
\lambda :0\leq \left\vert \lambda \right\vert \leq \lambda _{0}$; moreover,
the inclusion $B_{\delta _{0}}^{Y}\left( 0\right) \subseteq f_{\lambda
}\left( B_{r_{0}}^{X}\left( 0\right) \right) $ holds for $\delta _{0}\equiv
\delta _{0}\left( \lambda \right) >0$ from the condition 2, where $%
V_{\lambda }$ defined as follows 
\begin{equation*}
V_{\lambda }\equiv \left\{ \left. \widetilde{y}\in Y\right\vert \
\left\langle \widetilde{y},g\left( x\right) \right\rangle \leq \left\langle
f_{\lambda }\left( x\right) ,g\left( x\right) \right\rangle ,\quad \forall
x\in S_{r_{0}}^{X}\left( 0\right) \right\} .
\end{equation*}
\end{theorem}

For the proof is sufficient to note that all conditions of Theorem \ref{Th_1}
are fulfilled for each fixed $\lambda :$ $\left\vert \lambda \right\vert
<\lambda _{0}$ due to conditions of Theorem \ref{Th_3}, therefore applying
Theorem \ref{Th_1} we get the correctness of Theorem \ref{Th_3}.

Consequently, the equation (\ref{2.3}) also is solvable in $B_{r}^{X}\left(
0\right) $ under the conditions on $\widetilde{f}_{\lambda }$ of the above
type that depends at $\lambda _{0}$, e.g. 
\begin{equation*}
\left\Vert G\left( F^{-1}\left( y_{1}\right) \right) -G\left( F^{-1}\left(
y_{2}\right) \right) \right\Vert _{Y}\leq C\left( x_{0},\varepsilon \right)
\left\Vert y_{1}-y_{2}\right\Vert _{Y}+\psi _{\lambda }(\left\Vert
y_{1}-y_{2}\right\Vert _{Z},y_{0},\varepsilon ),
\end{equation*}%
where $C\left( x_{0},\varepsilon \right) \lambda _{0}<1$ and the inclusion $%
Y\subset Z$ is compact.

Whence using Theorem \ref{Th_2} one can obtain the solvability of the
equation (\ref{2.1a}). Really, let $Y=X^{\ast }$ and closed ball $%
B_{r_{0}}^{X}\left( x_{0}\right) $ ($r_{0}>0$) belongs to $D\left( F\right) $%
. Let condition 1 of Theorem \ref{Th_3} is fulfilled on ball $%
B_{r_{0}}^{X}\left( x_{0}\right) $. Assume the following conditions are
fulfilled.

2%
\'{}%
) There exists such parameter $\lambda _{1}\in 
\mathbb{R}
$ that $\lambda _{1}G\left( F^{-1}\left( F\left( B_{r_{0}}^{X}\left(
x_{0}\right) \right) \right) \right) \subseteq B_{r_{0}}^{X}\left(
x_{0}\right) $ and for each $x^{\ast }\in S_{1}^{X^{\ast }}\left( 0\right) $%
\ there exists such $x\in S_{r}^{X}\left( x_{0}\right) $, for each $r\in
\left( 0<r\leq r_{0}\right] $ that the following inequality 
\begin{equation*}
\left\langle \widetilde{f}_{\lambda _{1}}\left( x\right) ,x^{\ast
}\right\rangle \geq \nu _{\lambda _{1}}\left( \left\Vert x-x_{0}\right\Vert
_{X}\right) =\nu _{\lambda _{1}}\left( r\right) ,\quad x\in S_{r}^{X}\left(
x_{0}\right) \subset B_{r_{0}}^{X}\left( x_{0}\right)
\end{equation*}%
holds, where $\nu _{\lambda _{1}}:\Re _{+}\longrightarrow \Re $ is the
continuous function that satisfies the condition $\left( ii\right) $ of
Theorem \ref{Th_1}

3%
\'{}%
) For almost every $\widehat{x}\in B_{r_{0}}^{X}\left( x_{0}\right) $ there
are such numbers $\varepsilon \geq \varepsilon _{0}>0$ and continuous
functions $\Phi _{\lambda _{1}}(\cdot ,\widehat{x},\varepsilon ):\Re
_{+}\longrightarrow $ $\Re _{+}$, $\varphi _{\lambda }(\cdot ,\widehat{x}%
,\varepsilon ):\Re _{+}\longrightarrow $ $\Re $ for each $\left( \widehat{x}%
,\varepsilon \right) $ that the following inequality 
\begin{equation*}
\left\Vert \widetilde{f}_{\lambda _{1}}\left( x_{1}\right) -\widetilde{f}%
_{\lambda _{1}}\left( x_{2}\right) \right\Vert _{X}\geq \Phi _{\lambda
_{1}}\left( \left\Vert x_{1}-x_{2}\right\Vert _{X},\widehat{x},\varepsilon
\right) +\varphi _{\lambda _{1}}\left( \left\Vert x_{1}-x_{2}\right\Vert
_{Z},\widehat{x},\varepsilon \right)
\end{equation*}%
holds for any $x_{1},x_{2}\in U_{\varepsilon }\left( \widehat{x}\right) \cap
B_{r_{0}}^{X}\left( x_{0}\right) $, where $\Phi _{\lambda _{1}}\left( \tau ,%
\widehat{x},\varepsilon \right) \geq 0$ and $\Phi _{\lambda _{1}}\left( \tau
,\widehat{x},\varepsilon \right) =0\Leftrightarrow \tau =0$, $\varphi
_{\lambda _{1}}\left( 0,\widehat{x},\varepsilon \right) =0$, and also $Z$ is
the Banach space such that $X\subset Z$ is compact.

Whence implies, that for defined above $\lambda _{1}$ all conditions of
Theorem \ref{Th_1} are fulfilled for the operator $\widetilde{f}_{\lambda
_{1}}$ on the closed ball $B_{r_{0}}^{X}\left( x_{0}\right) $. Consequently, 
$\widetilde{f}_{\lambda _{1}}\left( B_{r_{0}}^{X}\left( x_{0}\right) \right) 
$ contains a closed absorbing subset of $X$ (at least, $0\in X$) by virtue
of the Theorem \ref{Th_1}. In the other words, $0\in \widetilde{f}_{\lambda
_{1}}\left( B_{r_{0}}^{X}\left( x_{0}\right) \right) $ and therefore there
exists an element $\widetilde{x}\in B_{r_{0}}^{X}\left( x_{0}\right) $ for
which $f_{\lambda _{1}}\left( \widetilde{x}\right) =0$ holds, i. e. $F\left( 
\widetilde{x}\right) =\lambda _{1}G\left( \widetilde{x}\right) $.

The obtained result one can formulate as follows.

\begin{corollary}
\label{C_2} Let $F,G$ be above determined operators, $F\succ G$, $D\left(
F\right) \subseteq D\left( G\right) $, and there exists such number $\lambda
_{1}$ that conditions 1), 2%
\'{}%
), 3%
\'{}%
) are fulfilled on the closed ball $B_{r}^{X}\left( x_{0}\right) \subseteq
D\left( F\right) \subseteq X$. Then there exists an element $\widetilde{x}%
\in B_{r}^{X}\left( x_{0}\right) $ such, that $F\left( \widetilde{x}\right)
=\lambda _{1}G\left( \widetilde{x}\right) $ or $\lambda _{1}G\left(
F^{-1}\left( \cdot \right) \right) $ has fixed point.
\end{corollary}

Let $X,Y$ be Banach spaces, and\ $B_{r_{0}}^{X}\left( 0\right) \subseteq
D\left( F\right) \subset X$, $r_{0}>0$, $F\succ G$, $F\left( 0\right) =0$, $%
G\left( 0\right) =0$ be a bounded operators, and there are the continuous
functions $\nu _{F},\nu _{G}:\Re _{+}\longrightarrow \Re $ satisfying
condition $\left( ii\right) $ of Theorem \ref{Th_1} such that for each $%
y^{\ast }\in S_{1}^{Y^{\ast }}\left( 0\right) $\ there exists $x\in
S_{r}^{X}\left( 0\right) $, for which the inequalities \ 
\begin{equation*}
\left\langle F\left( x\right) ,y^{\ast }\right\rangle \geq \nu _{F}\left(
\left\Vert x\right\Vert _{X}\right) ,\quad \left\langle G\left( x\right)
,y^{\ast }\right\rangle \geq \nu _{G}\left( \left\Vert x\right\Vert
_{X_{1}}\right)
\end{equation*}%
hold, where $X_{1}$ is the Banach space that $X\subseteq X_{1}$ (we will
denote the relation between $x$ and $y^{\ast }\in S_{1}^{Y^{\ast }}\left(
0\right) $ by $g:S_{r}^{X}\left( 0\right) \longrightarrow S_{1}^{Y^{\ast
}}\left( 0\right) $, $0<r\leq $\ $r_{0}$, so that $g\left( x\right) =y^{\ast
}$). Then according to condition 2%
\'{}
of Corollary \ref{C_2}, one may expect that spectrum of the operator $%
F:D\left( F\right) \subset X\longrightarrow Y$ relative to operator $%
G:D\left( G\right) \subseteq X\longrightarrow Y$ can define in the following
way 
\begin{equation}
\lambda =\inf \left\{ \frac{\left\langle F\left( x\right) ,g\left( x\right)
\right\rangle }{\left\langle G\left( x\right) ,g\left( x\right)
\right\rangle }\left\vert \ x\in \right. B_{r_{0}}^{X}\left( 0\right)
\backslash \left\{ 0\right\} \right\} ,\quad r_{0}>0.  \label{2.5}
\end{equation}

Now we will clarify when the determined by (\ref{2.5}) $\lambda $ is the
spectrum of operator $G\circ F^{-1}$ or the spectrum of operator $F$
relative to operator $G$. Generally speaking, one cannot name it since the
composition $G\circ F^{-1}$ can be nonlinear and $\lambda _{1}$ may be a
function as $\lambda _{1}=\lambda _{1}\left( x_{1}\right) $, unlike the
linear case, where $x_{1}$ is the element on which the relation (\ref{2.5})
attains the infimum. Moreover, if we define the subspace $\Gamma _{\lambda
_{1}}=\left\{ \alpha x_{1}\left\vert \ \alpha \in R\right. \right\} \subset
X $ then for $\alpha x_{1}\in D\left( F\right) $, generally, $\alpha \lambda
_{1}x_{1}\neq G\circ F^{-1}\left( \alpha x_{1}\right) $ since $G\circ F^{-1}$
is the nonlinear operator.

Indeed, if the power of nonlinearity of the operator $F$ is great than the
power of nonlinearity of operator $G$, or the inverse of its, then
obviously, will the case $\lambda _{1}=\lambda _{1}\left( x_{1}\right) $.
For example, operators $F$ and $G$ defines in the following form 
\begin{equation*}
F\left( u\right) =-\nabla \circ \left( \left\vert \nabla u\right\vert
^{p_{0}-2}\nabla u\right) ,\quad G\left( u\right) =\left\vert u\right\vert
^{p_{1}-2}u,\quad Y=W^{-1,q}\left( \Omega \right) ,
\end{equation*}%
where $X=W_{0}^{1,p_{0}}\left( \Omega \right) \cap L^{p_{1}}\left( \Omega
\right) $, $\Omega \subset R^{n}$, $n\geq 1$, with sufficiently smooth
boundary $\partial \Omega $ and $p=\max \left\{ p_{0},p_{1}\right\} $, $%
p_{0},p_{1}>2$, $q=p`=\frac{p}{p-1}$. \ Assume $p_{0}\neq p_{1}$ and $%
F:D\left( F\right) =W^{1,p_{0}}\left( \Omega \right) \longrightarrow $ $%
W^{-1,q_{0}}\left( \Omega \right) $, $G:D\left( G\right) =L^{p_{1}}\left(
\Omega \right) \longrightarrow $ $L^{q_{1}}\left( \Omega \right) $. Then
using (\ref{2.5}) we get 
\begin{equation*}
\lambda =\inf \left\{ \frac{\left\langle F\left( u\right) ,u\right\rangle }{%
\left\langle G\left( u\right) ,u\right\rangle }\left\vert \ u\in \right.
B_{r_{0}}^{X}\left( 0\right) \backslash \left\{ 0\right\} \right\} =
\end{equation*}%
\begin{equation*}
\inf \left\{ \frac{\left\Vert \nabla u\right\Vert _{L^{p_{0}}}^{p_{0}}}{%
\left\Vert u\right\Vert _{L^{p_{1}}}^{p_{1}}}\left\vert \ u\in
B_{r_{0}}^{W_{0}^{1,p_{0}}\cap L^{p_{1}}\left( \Omega \right) }\left(
0\right) \backslash \left\{ 0\right\} \right. \right\} .
\end{equation*}%
Whence we have if $p_{0}>p_{1}$, then 
\begin{equation*}
\lambda =\inf \left\{ \left( \frac{\left\Vert \nabla u\right\Vert
_{L^{p_{0}}}}{\left\Vert u\right\Vert _{L^{p_{1}}}}\right)
^{p_{1}}\left\Vert \nabla u\right\Vert _{L^{p_{0}}}^{p_{0}-p_{1}}\left\vert
\ u\in B_{r_{0}}^{W_{0}^{1,p_{0}}\cap L^{p_{1}}\left( \Omega \right) }\left(
0\right) \backslash \left\{ 0\right\} \right. \right\}
\end{equation*}%
and if $p_{0}<p_{1}$, then 
\begin{equation*}
\lambda =\inf \left\{ \left( \frac{\left\Vert \nabla u\right\Vert
_{L^{p_{0}}}}{\left\Vert u\right\Vert _{L^{p_{1}}}}\right)
^{p_{0}}\left\Vert u\right\Vert _{L^{p_{1}}}^{p_{0}-p_{1}}\left\vert \ u\in
B_{r_{0}}^{W_{0}^{1,p_{0}}\cap L^{p_{1}}\left( \Omega \right) }\left(
0\right) \backslash \left\{ 0\right\} \right. \right\} .
\end{equation*}

Consequently, $\lambda $ will be a function of $u$, i.e. $\lambda
_{1}=\lambda \left( u_{1}\right) $, here $u_{1}\in X$ is the element of the
domain, on which the above-mentioned expression attained the infimum (see,
still examples in Section 3.).

\begin{remark}
\label{R_1} Whence implies in order to $\lambda _{0}$ couldn't be the
function of $x$ main parts of operators $F$ and $G$ must have a common
degree of nonlinearity. Theorem \ref{Th_1} follows that the defined number $%
\lambda _{0}$ is the number assumed to exist in the conditions of this
theorem. Moreover, the finding number allows us in mentioned theorem to
state the existence of solutions for each $\lambda :0\leq \left\vert \lambda
\right\vert \leq \lambda _{0}$ if the element on the right side is from a
determined subset.
\end{remark}

The spectrum of an operator usually must be to characterize the examined
operator, but the found $\lambda $ doesn't satisfy this. Therefore, here is
used another way, different from the above-mentioned works.

Due to the above explanations, we get at that in order to the found $\lambda 
$ not be a function on $x$, is necessary for the existence of some relations
between the operators $F$ and $G$.

So, we assume one of the following conditions are fulfilled: (1) $F$ and $G$
are homogeneous with common exponent $p>0$ or a common function $\phi \left(
\cdot \right) $, i.e. $F\left( \mu x\right) \equiv \mu ^{p}F\left( x\right) $%
, $G\left( \mu x\right) \equiv \mu ^{p}G\left( x\right) $ or $F\left( \mu
x\right) \equiv \phi \left( \mu \right) F\left( x\right) $, $G\left( \mu
x\right) \equiv \phi \left( \mu \right) G\left( x\right) $ for any $\mu >0$;
(2) Investigate the problem locally, i.e. study the problem on the closed
ball $B_{r}^{X}\left( 0\right) \subseteq D\left( f_{\lambda }\right) $ for
selected $r>0$ and to seek of $\lambda $ in the form $\lambda \equiv \lambda
\left( r\right) $. \ 

We start to study case (1), i.e. when $F$ and $G$ are homogeneous with
exponent $p>0$. Whence implies that (\ref{2.5}) defines the number $\lambda $
independent on $x$, hence if denote this minimum by $\lambda _{1}$ and the
element, at which minimum attained by $x_{1}$, then (\ref{2.5}) will fulfill
for $\forall x\in $ $\Gamma _{\lambda _{1}}\cap D\left( F\right) $.
Consequently, in this case, one can define $x_{1}$ as the first eigenvector
and $\lambda _{1}$ as the first eigenvalue of operator $F$ relative to
operator $G$ (as in the linear case). In other words, $y_{1}=F\left(
x_{1}\right) $ is the fixed point of operator $\lambda G\circ F^{-1}$.

Now let be the case (2). Then if the orders of homogeneity $F$ and $G$ are
different, i.e. that given by different functions, e.g. by polynomial
functions with exponents $p_{F}\neq p_{G}$ then possible 2 variants: \textit{%
(a) } $p_{F}>p_{G}$ and \textit{(b) }$p_{F}<p_{G}$.

If the case (a) occur then $F\left( x\right) =r^{p_{F}}F\left( \widetilde{x}%
\right) $ and $G\left( x\right) =r^{p_{G}}G\left( \widetilde{x}\right) $ if
any $x\in X_{0}$ to write as $x\equiv r\widetilde{x}\ $ where $\left\Vert
x\right\Vert _{X_{0}}\equiv r$ and $\widetilde{x}=\frac{x}{r}\in
S_{1}^{X_{0}}\left( 0\right) \subset X_{0}$. Hence due to Theorem \ref{Th_3}%
, we get that $G$ can be the perturbation of operator $F$, therefore this
case not is essential. Let be the case (b). In this case, if there exists
such $\lambda _{0}$ and $x_{0}$ that $F\left( x_{0}\right) =\lambda
_{0}G\left( x_{0}\right) $ then 
\begin{equation*}
F\left( x_{0}\right) =r_{0}^{p_{F}}F\left( \widetilde{x}_{0}\right) ,\
G\left( x\right) =r_{0}^{p_{G}}G\left( \widetilde{x}_{0}\right)
\Longrightarrow r_{0}^{p_{F}}F\left( \widetilde{x}_{0}\right) =\lambda
_{0}r_{0}^{p_{F}}G\left( \widetilde{x}_{0}\right)
\end{equation*}%
\begin{equation*}
\Longrightarrow F\left( \widetilde{x}_{0}\right) =\lambda \left( \lambda
_{0},r_{0}\right) G\left( \widetilde{x}_{0}\right) \Longrightarrow \lambda
\left( \lambda _{0},r_{0}\right) =\lambda _{0}r_{0}^{p_{_{G}}-p_{_{F}}}
\end{equation*}%
holds. Hence follows that if we change $x_{0}\equiv r_{0}\ \widetilde{x}_{0}$
to $x_{1}\equiv r_{1}\ \widetilde{x}_{0}$, then $\lambda $ will change to $%
\lambda =\lambda _{0}r_{1}^{p_{_{G}}-p_{_{F}}}$. In other words, where
follows if $p_{F}\neq p_{G}$ then any existing number $\lambda $ will depend
on element $x\in X$, i.e. $\lambda =\lambda \left( r\right) $ on the line $%
\left\{ x\in X\left\vert \ x=r\ \widetilde{x}_{0},\right. r\in 
\mathbb{R}
\right\} $. The previous discussion shows, that there are two variants
either $p_{F}=p_{G}$ or $\lambda _{0}=\lambda _{0}\left( x_{0}\right) $, and
investigating these cases will be sufficient.

So, here we will study the posed question mainly in the case when condition $%
p_{F}=p_{G}$ holds.

Consequently, the concept defined in the articles \cite{1, 2, 3, 11, 13, 18,
19, 22, 25, 27, 28}, etc. of the semilinear spectral set is special case of
the Definition \ref{D_S}, by virtue of (\ref{2.2}) and (\ref{2.3}).

\section{Some Application of General Results}

Consider the following problems 
\begin{equation}
-\nabla \circ \left( \left\vert \nabla u\right\vert ^{p-2}\nabla u\right)
-\lambda \left\vert u\right\vert ^{p_{0}-2}u\left\vert \nabla u\right\vert
^{p_{1}}=0,\quad u\left\vert _{\ \partial \Omega }\right. =0,\ \lambda \in 
\mathbb{C}
,  \label{3.1}
\end{equation}

\begin{equation}
-\nabla \circ \left( \left\vert u\right\vert ^{p-2}\nabla u\right) -\lambda
\left\vert u\right\vert ^{p_{0}-2}u=0,\quad u\left\vert _{\ \partial \Omega
}\right. =0,\ \lambda \in 
\mathbb{C}
,  \label{3.2}
\end{equation}%
where $\Omega \subset 
\mathbb{R}
^{n}$ is an open bounded domain with sufficiently smooth boundary $\partial
\Omega $, $n\geq 1$, $p_{0}+p_{1}=p$ and $\nabla \equiv \left(
D_{1},...,D_{n}\right) $. Denote by $f_{0}$ the operator generated by ((\ref%
{3.1}) which acts from $W_{0}^{1.p}\left( \Omega \right) $ to $%
W^{-1,q}\left( \Omega \right) $. It is easy to see that $f_{0}:W_{0}^{1.p}%
\left( \Omega \right) \longrightarrow W^{-1,q}\left( \Omega \right) $ is a
continuous operator, and 
\begin{equation*}
0=\left\langle f_{0}\left( u\right) ,u\right\rangle \equiv \left\langle
-\nabla \circ \left( \left\vert \nabla u\right\vert ^{p-2}\nabla u\right)
-\lambda \left\vert u\right\vert ^{p_{0}-2}u\left\vert \nabla u\right\vert
^{p_{1}},u\right\rangle =
\end{equation*}%
\begin{equation*}
\left\Vert \nabla u\right\Vert _{p}^{p}-\underset{\Omega }{\int }\lambda
\left\vert u\right\vert ^{p_{0}}\left\vert \nabla u\right\vert ^{p_{1}}dx\
\Longrightarrow \left\Vert \nabla u\right\Vert _{p}^{p}=\lambda \underset{%
\Omega }{\int }\left\vert u\right\vert ^{p_{0}}\left\vert \nabla
u\right\vert ^{p_{1}}dx\ 
\end{equation*}%
holds for any $u\in W_{0}^{1.p}\left( \Omega \right) $. Whence follows $%
\lambda \geq 0$ since both of these expressions are positive.

(1) We will investigate of problem (\ref{3.1}) by using of the Theorem \ref%
{Th_1} or Theorem \ref{Th_3} but we interest to study the question on the
spectrum therefore here we will use Corollary \ref{C_2}. According to the
previous section, we can introduce the following denotations 
\begin{equation*}
F\left( u\right) =-\nabla \left( \left\vert \nabla u\right\vert ^{p-2}\nabla
u\right) ,\quad F:W_{0}^{1.p}\left( \Omega \right) \longrightarrow
W^{-1,q}\left( \Omega \right) ,
\end{equation*}%
\begin{equation*}
G\left( u\right) =\left\vert u\right\vert ^{p_{0}-2}u\left\vert \nabla
u\right\vert ^{p_{1}},\quad G:W_{0}^{1.p}\left( \Omega \right)
\longrightarrow W^{-1,q}\left( \Omega \right) .
\end{equation*}

So, it needs to seek the minimal value of $\lambda $ and function $%
u_{\lambda }\left( x\right) $ (if it exists) for which the equality 
\begin{equation*}
\left\Vert \nabla u\right\Vert _{p}^{p}=\lambda \underset{\Omega }{\int }%
\left\vert u\right\vert ^{p_{0}}\left\vert \nabla u\right\vert
^{p_{1}}dx,\quad u\in B_{1}^{W_{0}^{1.p}\left( \Omega \right) }\left(
0\right)
\end{equation*}%
or holds the equality 
\begin{equation*}
\lambda =\frac{\left\Vert \nabla u\right\Vert _{p}^{p}}{\underset{\Omega }{%
\int }\left\vert u\right\vert ^{p_{0}}\left\vert \nabla u\right\vert
^{p_{1}}dx}=\ \underset{\Omega }{\int }\left( \frac{\left\Vert \nabla
u\right\Vert _{p}}{\left\vert u\right\vert }\right) ^{p_{0}}\left( \frac{%
\left\Vert \nabla u\right\Vert _{p}}{\left\vert \nabla u\right\vert }\right)
^{p_{1}}dx.
\end{equation*}%
Consequently, we need to find the following number 
\begin{equation}
\lambda _{1}=\inf \left\{ \frac{\left\Vert \nabla u\right\Vert _{p}^{p}}{%
\underset{\Omega }{\int }\left\vert u\right\vert ^{p_{0}}\left\vert \nabla
u\right\vert ^{p_{1}}dx}\left\vert \ u\in B_{1}^{\overset{0}{W}\
^{1.p}}\left( 0\right) \right. \right\} .  \label{3.3a}
\end{equation}

It is clear that $\lambda _{1}$ exists and $\lambda _{1}>0$.

Whence follows 
\begin{equation}
\lambda _{1}\geq \left( \frac{\left\Vert \nabla u\right\Vert _{p}}{%
\left\Vert u\right\Vert _{p}}\right) ^{p_{0}}\Longrightarrow \lambda _{1}^{%
\frac{1}{p_{0}}}\geq \frac{\left\Vert \nabla u\right\Vert _{p}}{\left\Vert
u\right\Vert _{p}}  \label{3.3}
\end{equation}%
for any $u\in W_{0}^{1.p}\left( \Omega \right) $, $u\left( x\right) \neq 0$
that is assumes in what follows.

We denote by $\lambda _{p_{0},p_{1}}$ the first spectrum of the posed
problem that one can define as 
\begin{equation}
\lambda _{p_{0},p_{1}}=\inf \left\{ \left\Vert \nabla u\right\Vert _{p}\left[
\underset{\Omega }{\int }\left\vert u\right\vert ^{p_{0}}\left\vert \nabla
u\right\vert ^{p_{1}}dx\right] ^{-\frac{1}{p}}\left\vert \ u\in
S_{1}^{W_{0}^{1.p}\left( \Omega \right) }\left( 0\right) \right. \right\} .
\label{3.4}
\end{equation}%
From (\ref{3.3}) we obtain 
\begin{equation*}
\lambda _{p_{0},p_{1}}\geq \lambda _{1}^{\frac{1}{p_{0}}}=\inf \left\{ \frac{%
\left\Vert \nabla u\right\Vert _{p}}{\left\Vert u\right\Vert _{p}}\
\left\vert \ u\in W_{0}^{1.p}\left( \Omega \right) \right. \right\}
\end{equation*}%
that is well-known $\lambda _{1}^{\frac{1}{p_{0}}}=\lambda _{1}\left(
-\Delta _{p}\right) $ was defined as the first spectrum of $p-$Laplacian
(see, e.g. \cite{23}) that is

\begin{equation*}
\lambda _{1}\left( -\Delta _{p}\right) =\inf \ \left\{ \frac{\left\Vert
\nabla u\right\Vert _{p}}{\left\Vert u\right\Vert _{p}}\ \left\vert \ u\in
W_{0}^{1.p}\left( \Omega \right) \right. \right\} ,
\end{equation*}%
consequently, this inequality shows that $\lambda _{p_{0},p_{1}}$ is
comparable with the spectrum $\lambda _{1}\left( -\Delta _{p}\right) $ of
the $p-$Laplacian, i.e. $\lambda _{p_{0},p_{1}}$ satisfy the inequality $%
\lambda _{p_{0},p_{1}}\geq $ $\lambda _{1}\left( -\Delta _{p}\right) $.

(2) Now we will consider the problem (\ref{3.2}) then it is getting 
\begin{equation*}
\underset{\Omega }{\int }\left\vert u\right\vert ^{p-2}\left\vert \nabla
u\right\vert ^{2}dx=\lambda \underset{\Omega }{\int }\left\vert u\right\vert
^{p_{0}}dx
\end{equation*}%
or 
\begin{equation*}
\frac{4}{p^{2}}\left\Vert \nabla \left( \left\vert u\right\vert ^{\frac{p-2}{%
2}}u\right) \right\Vert _{2}^{2}=\lambda \left\Vert \left\vert u\right\vert
^{\frac{p_{0}-2}{2}}u\right\Vert _{2}^{2},
\end{equation*}%
here if assume $p_{0}=p$ and $\left\vert u\right\vert ^{\frac{p-2}{2}%
}u\equiv v$ then we get 
\begin{equation*}
\frac{4}{p^{2}}\left\Vert \nabla v\right\Vert _{2}^{2}=\frac{4}{p^{2}}%
\left\Vert \nabla \left( \left\vert u\right\vert ^{\frac{p-2}{2}}u\right)
\right\Vert _{2}^{2}=\lambda \left\Vert \left\vert u\right\vert ^{\frac{p-2}{%
2}}u\right\Vert _{2}^{2}=\lambda \left\Vert v\right\Vert _{2}^{2}.
\end{equation*}

Whence follows, that the first eigenvalue $\lambda _{1}\left( p\right) $ of
the operator $-\nabla \cdot \left( \left\vert u\right\vert ^{p-2}\nabla
u\right) $ relative to operator $\left\vert u\right\vert ^{p-2}u$ can be
defined using the first eigenvalue $\lambda _{1}\left( -\Delta \right) $ of
the Laplacian that is defined by expression 
\begin{equation*}
\lambda _{1}\left( -\Delta \right) =\inf \ \left\{ \frac{\left\Vert \nabla
v\right\Vert _{2}}{\left\Vert v\right\Vert _{2}}\ \left\vert \ v\in
W_{0}^{1.2}\left( \Omega \right) \right. \right\} .
\end{equation*}%
Consequently, the first eigenvalue $\lambda _{1}\left( p\right) $ of the
operator $-\nabla \cdot \left( \left\vert u\right\vert ^{p-2}\nabla u\right) 
$ relative to operator $\left\vert u\right\vert ^{p-2}u$ (in appropriate to
problem (\ref{3.2})) one can define by the equality $\lambda _{1}\left(
p\right) =\left( \frac{2}{p}\lambda _{1}\left( -\Delta \right) \right) ^{2}$.

Thus we get

\begin{proposition}
\label{Pr_1} (1) Let $\Delta _{p}$ be the $p-$Laplacian operator with
homogeneous boundary conditions on the bounded domain of $R^{n}$ with smooth
boundary $\partial \Omega $ and $p_{0}+p_{1}=p$. Then the first eigenvalue $%
\lambda _{p_{0}p_{1}}$ for the operator $-\Delta _{p}$ relative to operator $%
G:G\left( u\right) \equiv \left\vert u\right\vert ^{p_{0}-2}u\left\vert
\nabla u\right\vert ^{p_{1}}$ exists and be defined with equality (\ref{3.4}%
).

(2) If $F$ and $G$ be operators are generated by the problem (\ref{3.2}) and 
$p_{0}=p$, then the first eigenvalue of the operator $F$ relative to
operator $G$ is determined as $\lambda _{1}\left( p\right) =\left( \frac{2}{p%
}\lambda _{1}\left( -\Delta \right) \right) ^{2}$, where $\lambda _{1}\left(
-\Delta \right) $ is the first eigenvalue of the Laplacian.
\end{proposition}

\begin{remark}
\label{R_2} It should be noted by the same way one can define the spectrum
of operator $\Delta _{p}$ (and also of the operator $F:F\left( u\right)
\equiv \underset{i=1}{\overset{n}{\sum }}D_{i}\left( \left\vert
D_{i}u\right\vert ^{p-2}D_{i}u\right) $) relative to operator $%
G_{0}:G_{0}\left( u\right) \equiv \underset{i=1}{\overset{n}{\sum }}%
\left\vert u\right\vert ^{p_{0}-2}u\left\vert D_{i}u\right\vert ^{p_{1}}$.
\end{remark}

Now, using the previous results we will investigate the solvability of the
problem on the open bounded domain $\Omega \subset R^{n}$ with smooth
boundary $\partial \Omega $ with the following equation whit the homogeneous
boundary condition 
\begin{equation}
f_{\lambda }\left( u\right) \equiv -\nabla \left( \left\vert \nabla
u\right\vert ^{p-2}\nabla u\right) -\lambda \left\vert u\right\vert
^{p_{0}-2}u\left\vert \nabla u\right\vert ^{p_{1}}=h\left( x\right)
\label{3.5}
\end{equation}%
For this problem the following result holds.

\begin{theorem}
\label{Th_4} Let numbers $p$, $p_{0}$, $p_{1}\geq 0$ are such that $%
p_{0}+p_{1}=p\geq 2$, $\lambda _{p_{0},p_{1}}$ is the number defined in (\ref%
{3.4}). Then if $\lambda <\lambda _{p_{0},p_{1}}$ then the posed problem is
solvable in $W_{0}^{1.p}\left( \Omega \right) $ for each $h\in
W^{-1,q}\left( \Omega \right) $.
\end{theorem}

\begin{proof}
Let us 
\begin{equation*}
f_{\lambda }\left( u\right) \equiv -\nabla \left( \left\vert \nabla
u\right\vert ^{p-2}\nabla u\right) -\lambda \left\vert u\right\vert
^{p_{0}-2}u\left\vert \nabla u\right\vert ^{p_{1}}
\end{equation*}%
be the operator generated by the posed problem for (\ref{3.5}) acting as $%
f_{\lambda }:X\longrightarrow Y$, where $X\equiv W_{0}^{1.p}\left( \Omega
\right) $, $Y\equiv W^{-1,q}\left( \Omega \right) $ and fulfill the above
conditions. We will use the Corollary \ref{C_1} for the proof of solvability
of this problem.

Whence follows that the inequality \ 
\begin{equation*}
\left\langle f_{\lambda }\left( u\right) ,u\right\rangle \equiv \left\Vert
\nabla u\right\Vert _{p}^{p}-\lambda \underset{\Omega }{\int }\left\vert
u\right\vert ^{p_{0}}\left\vert \nabla u\right\vert ^{p_{1}}dx\geq
\end{equation*}%
\begin{equation*}
\left\Vert \nabla u\right\Vert _{p}^{p}-\lambda \left\Vert u\right\Vert
_{p}^{p_{0}}\left\Vert \nabla u\right\Vert _{p}^{p_{1}}=\left\Vert \nabla
u\right\Vert _{p}^{p_{1}}\left( \left\Vert \nabla u\right\Vert
_{p}^{p_{0}}-\lambda \left\Vert u\right\Vert _{p}^{p_{0}}\right) =
\end{equation*}%
\begin{equation*}
\left\Vert \nabla u\right\Vert _{p}^{p}\left( 1-\frac{\lambda }{\lambda
_{p_{0},p_{1}}}\right) =\lambda _{p_{0},p_{1}}^{-1}\left( \lambda
_{p_{0},p_{1}}-\lambda \right) \left\Vert \nabla u\right\Vert _{p}^{p}
\end{equation*}%
holds for any $u\in W_{0}^{1.p}\left( \Omega \right) $ under conditions of
Theorem \ref{Th_4}.

Consequently, if $\lambda <\lambda _{p_{0},p_{1}}$ then $f_{\lambda }$
satisfies the condition \textit{(ii)} of the Theorem \ref{Th_1}, moreover it
is fulfilled for $x_{0}=0$ and $g\equiv Id$. The realization of the
condition (i) of Theorem \ref{Th_1} for $f_{\lambda }$ is obvious.

The following inequalities show the fulfillment of condition \textit{(iii) }%
of the Theorem \ref{Th_1} (Corollary \ref{C_1}) for this problem. Isn't
difficult to see 
\begin{equation*}
\left\langle f\left( u\right) -f\left( v\right) ,u-v\right\rangle \equiv
\left\langle \left( \left\vert \nabla u\right\vert ^{p-2}\nabla u-\left\vert
\nabla v\right\vert ^{p-2}\nabla v\right) ,\nabla \left( u-v\right)
\right\rangle -
\end{equation*}%
\begin{equation*}
\lambda \left\langle \left( \left\vert u\right\vert ^{p_{0}-2}\left\vert
\nabla u\right\vert ^{p_{1}}u-\left\vert v\right\vert ^{p_{0}-2}\left\vert
\nabla v\right\vert ^{p_{1}}v\right) ,u-v\right\rangle \geq c_{0}\left\Vert
\nabla \left( u-v\right) \right\Vert _{p}^{p}-
\end{equation*}%
\begin{equation*}
\lambda \left\langle \left\vert u\right\vert ^{p_{0}-2}u\left( \left\vert
\nabla u\right\vert ^{p_{1}}-\left\vert \nabla v\right\vert ^{p_{1}}\right)
,u-v\right\rangle -\lambda \left\langle \left\vert \nabla v\right\vert
^{p_{1}}\left( \left\vert u\right\vert ^{p_{0}-2}u-\left\vert v\right\vert
^{p_{0}-2}v\right) ,u-v\right\rangle
\end{equation*}%
hold for any $u,v\in W_{0}^{1.p}\left( \Omega \right) $. Here the second
term of the right side one can estimate as 
\begin{equation*}
\left\vert \left\langle \left\vert u\right\vert ^{p_{0}-2}u\left( \left\vert
\nabla u\right\vert ^{p_{1}}-\left\vert \nabla v\right\vert ^{p_{1}}\right)
,u-v\right\rangle \right\vert \leq c_{1}\left\langle \left\vert u\right\vert
^{p_{0}-1}\left\vert \nabla \widetilde{u}\right\vert ^{p_{1}-1}\left\vert
\nabla u-\nabla v\right\vert ,\left\vert u-v\right\vert \right\rangle \leq
\end{equation*}%
\begin{equation*}
\varepsilon \left\Vert \nabla \left( u-v\right) \right\Vert _{p}^{p}+C\left(
\varepsilon \right) \left\Vert u\right\Vert _{p}^{\left( p_{0}-1\right)
p^{\prime }}\left\Vert \nabla \widetilde{u}\right\Vert _{p}^{\left(
p_{1}-1\right) p^{\prime }}\left\Vert u-v\right\Vert _{p}^{p^{\prime }},\
p^{\prime }=\frac{p}{p-1}.
\end{equation*}%
Thus we get that all of the conditions of Theorem \ref{Th_1} (the case of
Corollary \ref{C_1}) are fulfilled for the problem (\ref{3.5}).
Consequently, applying Theorem \ref{Th_1} we get the correctness of Theorem %
\ref{Th_4}.
\end{proof}

\section{Fully Nonlinear Operator}

Now we will study the question on the existence of the spectrum of the fully
nonlinear operator. Let $X,Y,Z$ are the real Banach spaces, the inclusion $%
Y\subset Z^{\ast }$is continuous and dense, where $Z^{\ast }$ is the dual
space of $Z$. Let $L:D\left( L\right) \subseteq X\longrightarrow Y$ is the
linear operator, $D\left( L\right) $ is dense in $X$; $f:D\left( f\right)
\subseteq Y\longrightarrow Z$ and $g:X\subseteq D\left( f\right)
\longrightarrow Z$ are nonlinear operators. Assume $L\left( D\left( L\right)
\right) \subseteq $ $D\left( f\right) $.

We wish to define the spectrum of the operator $f\circ L:D\left( L\right)
\subseteq X\longrightarrow Z$ with respect to the operator $g$, where $%
D\left( L\right) \subseteq D\left( g\right) $. (We should to noted the
nonlinearity of operator $g$ depends on the nature of the nonlinearity of
the operator $f$.)

We will consider problem

\begin{equation}
f_{\lambda }\left( x\right) \equiv f\left( Lx\right) -\lambda g\left(
x\right) =h,\quad h\in Z,\   \label{4.1a}
\end{equation}%
where $\lambda \in \mathbb{%
\mathbb{R}
}$ is a parameter and $h$ is an element of $Z$.

We will investigate the existence of the spectrum of the operator $f\circ L$
relative to operator $g$, and also the solvability of the equation (\ref%
{4.1a}) with a parameter.

We will call a $\lambda \in 
\mathbb{C}
$ the spectrum of the nonlinear operator if it characterizes the examined
operator similarly to the linear operator theory according to Definition \ref%
{D_S}. Whence follows that necessary to assume the identical homogeneity of
the nonlinearities of operators $f$ and $g$ according to the above
explanations.

So, we will study the following particular case that can explain the general
case. We will use the general results of articles \cite{29, 30} to
investigate posed problems. Let $B_{r_{0}}^{X}\left( 0\right) \subset
D\left( L\right) $, $r_{0}>0$.

Consider the following conditions:

1) There are such constants $c_{1},c_{2}>0$ that $c_{1}\left\Vert
x\right\Vert _{X}\geq \left\Vert Lx\right\Vert _{Y}\geq c_{2}\left\Vert
x\right\Vert _{X}$ for any $x\in D\left( L\right) \subseteq X$ moreover, $X$
and $Y$ be reflexive spaces and the inverse to $L$ is a compact operator;

2) The operator $f\circ L$ is greater than the operator $g$, i.e. $f\circ
L\succ g$ ; $f$ and $g$ as the functions are continuous, and satisfy the
following conditions: $f\left( t\right) \cdot t>0$ for $\forall t\in
R\setminus \left\{ 0\right\} $; $f\left( 0\right) =0$, $g\left( 0\right) =0$;

3) There is such number $\lambda _{0}>0$ that for each $z^{\ast }\in
S_{1}^{Z^{\ast }}\left( 0\right) $ there exist such $x\left( z^{\ast
}\right) \in S_{r}^{X}\left( 0\right) $, $0\leq r\leq r_{0}$ that the
inequality 
\begin{equation*}
\left\langle f_{\lambda }\left( x\right) ,z^{\ast }\right\rangle \equiv
\left\langle f\left( Lx\right) -\lambda g\left( x\right) ,z^{\ast
}\right\rangle \geq \nu \left( \left\Vert Lx\right\Vert _{Y},\lambda \right)
\end{equation*}%
holds for $\forall \lambda :\left\vert \lambda \right\vert <\lambda _{0}$,
where $\nu :R_{+}\longmapsto R$ is a continuous function, and there exists $%
\delta _{0}\left( \lambda \right) >0$ such $\nu \left( t,\lambda \right)
\geq \delta _{0}\left( \lambda \right) $ that for $\forall t=\left\Vert
Lx\right\Vert _{Y}$ when the variable $x$ moves over sphere $%
S_{r_{0}}^{X}\left( 0\right) $;

4) There exist such $\varepsilon _{0}>0$ and a neighborhood $U_{\varepsilon
}\left( x\right) $ of a. e. $x\in B_{r_{0}}^{X}\left( 0\right) \subseteq X$
that the following inequalities \qquad\ 
\begin{equation*}
\left\langle f\left( Lx_{1}\right) -f\left( Lx_{2}\right)
,Lx_{1}-Lx_{2}\right\rangle \geq l\left( x_{1},x_{2}\right) \left\Vert
Lx_{1}-Lx_{2}\right\Vert _{Y}^{2},
\end{equation*}
\begin{equation*}
\left\Vert g\left( x_{1}\right) -g\left( x_{2}\right) \right\Vert _{Z}\leq
l_{1}\left( x_{1},x_{2}\right) \left\Vert x_{1}-x_{2}\right\Vert _{X}
\end{equation*}%
hold for $\forall x_{1},x_{2}\in U_{\varepsilon }\left( x\right) $, where $%
l\left( x_{1},x_{2}\right) >0$, $l_{1}\left( x_{1},x_{2}\right) >0$ be
functionals that are bounded in the sense, similar of the Definition \ref%
{D_1}.

\begin{theorem}
\label{Th_5} Let conditions 1-4 be fulfilled, $D\left( L\right) =X$, and $%
\lambda :\left\vert \lambda \right\vert \leq \lambda _{0}$. Then equation (%
\ref{4.1a}) solvable for $\forall h\in Z$ satisfies the condition: for any $%
z^{\ast }\in S_{1}^{Z^{\ast }}\left( 0\right) $ there exists such $x\left(
z^{\ast }\right) \in S_{r_{0}}^{X}\left( 0\right) $ that the following
relation 
\begin{equation*}
\left\vert \left\langle h,z^{\ast }\right\rangle \right\vert \leq
\left\langle f_{\lambda }\left( x\right) ,z^{\ast }\right\rangle ,
\end{equation*}%
holds, in particular, $B_{\delta _{0}\left( \lambda \right) }^{Z}\left(
0\right) \subseteq M\left( \lambda \right) $ (or (\ref{4.1a}) solvable for $%
\forall h\in M\left( \lambda \right) \subseteq Z$, where the subset $M\left(
\lambda \right) $ is determined by the following way: 
\begin{equation*}
M\left( \lambda \right) \equiv \left\{ z\in Z\left\vert \ \left\vert
\left\langle z,z^{\ast }\right\rangle \right\vert \leq \nu \left( \left\Vert
Lx\right\Vert _{Y},\lambda \right) ,\ \right. \forall z^{\ast }\in
S_{1}^{Z^{\ast }}\left( 0\right) ,\exists x\left( z^{\ast }\right) \in
S_{r_{0}}^{X}\left( 0\right) \right\} .\text{)}
\end{equation*}
\end{theorem}

\begin{proof}
For the proof, it is sufficient to show that the examined operator satisfies
all conditions of the general result of Subsection 2.1. It is clear operator 
$F_{\lambda }\left( x\right) \equiv f\left( Lx\right) -\lambda g\left(
x\right) $ satisfies conditions (\textit{i}), (\textit{ii}) of the general
results with the $x_{0}=0$, according to conditions 1-4 (since $F_{\lambda
}\left( 0\right) =0$). Then remains to show fulfill of condition \textit{%
(iii),} and for this sufficiently to investigate the following expression 
\begin{equation*}
\left\Vert F_{\lambda }\left( x_{1}\right) -F_{\lambda }\left( x_{2}\right)
\right\Vert _{Z}=\left\Vert \left( f\left( Lx_{1}\right) -\lambda g\left(
x_{1}\right) \right) -\left( f\left( Lx_{2}\right) -\lambda g\left(
x_{2}\right) \right) \right\Vert _{Z}.
\end{equation*}%
Now we will prove this expression satisfies the following inequality 
\begin{equation*}
\left\Vert F_{\lambda }\left( x_{1}\right) -F_{\lambda }\left( x_{2}\right)
\right\Vert _{Z}\geq c\left( l\left( x_{1},x_{2}\right) \left\Vert
x_{1}-x_{2}\right\Vert _{X},\lambda \right) -c_{1}\left( l_{1}\left(
x_{1},x_{2}\right) \left\Vert x_{1}-x_{2}\right\Vert _{X_{0}},\lambda
\right) .
\end{equation*}%
We set the following expression: 
\begin{equation*}
\left\langle F_{\lambda }\left( x_{1}\right) -F_{\lambda }\left(
x_{2}\right) ,Lx_{1}-Lx_{2}\right\rangle =\left\langle \left( f\left(
Lx_{1}\right) -f\left( Lx_{2}\right) \right) ,Lx_{1}-Lx_{2}\right\rangle -
\end{equation*}%
\begin{equation*}
-\lambda \left\langle g\left( x_{1}\right) -g\left( x_{2}\right)
,Lx_{1}-Lx_{2}\right\rangle
\end{equation*}%
that is defined correctly since $F:D\left( L\right) \subseteq
X\longrightarrow Z$ and $L:D\left( L\right) \subseteq X\longrightarrow
Y\subset Z^{\ast }$, whence by carrying out certain necessary operations and
considering the conditions of this section we get 
\begin{equation*}
\left\langle F_{\lambda }\left( x_{1}\right) -F_{\lambda }\left(
x_{2}\right) ,Lx_{1}-Lx_{2}\right\rangle =\left\langle f\left( Lx_{1}\right)
-f\left( Lx_{2}\right) ,Lx_{1}-Lx_{2}\right\rangle -
\end{equation*}%
\begin{equation*}
-\left\langle \lambda g\left( x_{1}\right) -\lambda g\left( x_{2}\right)
,Lx_{1}-Lx_{2}\right\rangle \geq l\left( x_{1},x_{2}\right) \left\Vert
Lx_{1}-Lx_{2}\right\Vert _{Y}^{2}-
\end{equation*}%
\begin{equation*}
-\left\vert \lambda \right\vert \left\Vert g\left( x_{1}\right) -g\left(
x_{2}\right) \right\Vert _{Z}\left\Vert L\left( x_{1}-x_{2}\right)
\right\Vert _{Y}\geq l\left( x_{1},x_{2}\right) \left\Vert L\left(
x_{1}-x_{2}\right) \right\Vert _{Y}^{2}-
\end{equation*}%
\begin{equation}
-\left\vert \lambda \right\vert \left\Vert g\left( x_{1}\right) -g\left(
x_{2}\right) \right\Vert _{Z}\left\Vert L\left( x_{1}-x_{2}\right)
\right\Vert _{Y}  \label{4.2a}
\end{equation}%
according to condition 1). Now, again of taking into account condition 1) we
arrive to 
\begin{equation}
\left\vert \left\langle F_{\lambda }\left( x_{1}\right) -F_{\lambda }\left(
x_{2}\right) ,Lx_{1}-Lx_{2}\right\rangle \right\vert \leq \left\Vert
F_{\lambda }\left( x_{1}\right) -F_{\lambda }\left( x_{2}\right) \right\Vert
_{Z}\cdot \left\Vert Lx_{1}-Lx_{2}\right\Vert _{Y}.  \label{4.3a}
\end{equation}%
Thus using inequalities (\ref{4.2a}) and (\ref{4.3a}) and condition 4 we get
the following estimate 
\begin{equation}
\left\Vert F_{\lambda }\left( x_{1}\right) -F_{\lambda }\left( x_{2}\right)
\right\Vert _{Z}\geq l\left( x_{1},x_{2}\right) \left\Vert L\left(
x_{1}-x_{2}\right) \right\Vert _{Y}-\left\vert \lambda \right\vert
\left\Vert g\left( x_{1}\right) -g\left( x_{2}\right) \right\Vert _{Z}\geq
\label{4.4a}
\end{equation}%
\begin{equation*}
l\left( x_{1},x_{2}\right) \left\Vert L\left( x_{1}-x_{2}\right) \right\Vert
_{Y}-\left\vert \lambda \right\vert l_{1}\left( x_{1},x_{2}\right)
\left\Vert x_{1}-x_{2}\right\Vert _{X}.
\end{equation*}

So, conditions (\textit{i})-(\textit{iii}) of the general theorem are
fulfilled under conditions of this theorem, and from the above inequality
follows the fulfillment of condition \textit{(iv)} that to ensure the
closedness of the image of $F\left( B_{r_{0}}^{X}\left( 0\right) \right) $.
Consequently, the correctness of Theorem \ref{Th_5} follows from the general
theorem.
\end{proof}

\begin{remark}
\label{R_5} It should be noted the defined in Theorem \ref{Th_5} subset $%
M\left( \lambda \right) $ will decrease by increases in the number $%
\left\vert \lambda \right\vert \nearrow \lambda _{0}$. Moreover, the
above-mentioned articles actually sought numbers of the type $\lambda _{0}$
the existence assumed in the previous theorem.
\end{remark}

Here we will investigate the discovery of such numbers that are independent
of elements of the domain. Then, the founded number in such a way can be
called the first eigenvalue of the examined operator relative to the
different operators as in Definition \ref{D_S}.

In what follow we will use some results from the article Berger \cite{6}
(see, also \cite{34}), therefore here we provide these results from article 
\cite{6}.

\begin{definition}
\label{D_B}(\cite{6})Let $A:X\longrightarrow X^{\ast }$ be a variational
operator. Then $A$ is of class $I$ if: \ 

(i) $A$ is bounded, i.e. $\left\Vert A\left( x\right) \right\Vert \leq \mu
\left( \left\Vert x\right\Vert \right) ;$

(ii) $A$ is continuous from the strong topology of $X$ to the weak topology
of $X^{\ast };$

(III) Oddness, i.e. $A\left( -x\right) =-A\left( x\right) $ ;

iv) Coerciveness, i.e. $\overset{1}{\underset{0}{\int }}\ \left\langle
A\left( sx\right) ,x\right\rangle ds\nearrow \infty $ when $\left\Vert
x\right\Vert _{X}\nearrow \infty $;

(V) Monotonicity: $\left\langle A\left( x_{1}\right) -A\left( x_{2}\right)
,x_{1}-x_{2}\right\rangle >0$, for any $x_{1},x_{2}\in X$.
\end{definition}

\begin{lemma}
\label{L_B1}(\cite{6})Let $A$ be as variational operator of class $I$, then
a $\partial A_{R}$ 
\begin{equation*}
\partial A_{R}=\left\{ x\in X\left\vert \ \overset{1}{\underset{0}{\int }}\
\left\langle A\left( sx\right) ,x\right\rangle ds=R\text{\/}\right. \right\}
\end{equation*}%
is a closed, bounded set in $X$. Furthermore $\left\Vert x\right\Vert
_{X}\geq k(R)>0$ and $\partial A_{R}$ is a weakly closed, bounded convex
set, where $k(R)$ is a constant independent of $x\in \partial A_{R}$.
\end{lemma}

So, consider the homogeneous equation (\ref{4.1a}) in order to investigate
the existence of the necessary number $\lambda _{0}$.

\begin{proposition}
\label{Pr_2} Let $X\subset Y$ and dense in $Y$, $Z=Y^{\ast }$. Let
conditions 1), 2), 4) of Theorem \ref{Th_5} are fulfilled for this case and
operators $f$, $g$ as the functions are monotone odd functions then there
exist such $\lambda _{0}>0$ and $x_{\lambda _{0}}\in \partial
E^{B_{R_{0}}^{X}\left( 0\right) }\subset X$ that $F_{\lambda _{0}}\left(
x_{\lambda _{0}}\right) \equiv f\left( Lx_{\lambda _{0}}\right) -\lambda
_{0}g\left( x_{\lambda _{0}}\right) =0$ holds for some number $R_{0}\gg 1$,
where $\partial E^{B_{R_{0}}^{X}\left( 0\right) }$ be defined as follows 
\begin{equation*}
\partial E^{B_{R_{0}}^{X}\left( 0\right) }=\left\{ x\in B_{R_{0}}^{X}\left(
0\right) \subset X\left\vert \ \underset{0}{\overset{1}{\int }}\left\langle
f\left( sLx\right) ,Lx\right\rangle ds=R_{0}\right. \right\} ,
\end{equation*}%
\begin{equation*}
E^{B_{R_{0}}^{X}\left( 0\right) }=\left\{ x\in B_{R_{0}}^{X}\left( 0\right)
\subset X\left\vert \ \underset{0}{\overset{1}{\int }}\left\langle f\left(
sLx\right) ,Lx\right\rangle ds\leq R_{0}\right. \right\}
\end{equation*}%
Moreover, the condition similar to condition 3 of the above Theorem \ref%
{Th_5} satisfies.
\end{proposition}

\begin{proof}
It is clear that $L\left( B_{R_{0}}^{X}\left( 0\right) \right) $ is convex
due to the linearity of the operator $L$. From above Lemma \ref{L_B1}
follows that $E^{B_{R_{0}}^{X}\left( 0\right) }$ is a weakly closed, bounded
convex set and $\left\Vert x\right\Vert _{X}\geq k(r_{0})>0$, where $%
k(r_{0}) $ is a constant independent of $x\in E^{B_{R_{0}}^{X}\left(
0\right) }$, as the operator $f\circ L$ satisfies all conditions of Lemma %
\ref{L_B1}.

Consequently, Lemma \ref{L_B1} follows that $E^{B_{R_{0}}^{X}\left( 0\right)
}$ and \ $\partial E^{B_{R_{0}}^{X}\left( 0\right) }$ are weakly closed
bounded convex sets. Consequently, these are closed convex sets due to the
Mazur Theorem.

Now, consider the expression $\left\langle g\left( x\right) ,Lx\right\rangle 
$ and note that there exists such constant $M$ that 
\begin{equation*}
0<\sup \left\{ \left\Vert g\left( x\right) \right\Vert _{Y^{\ast
}}\left\vert \ x\in \partial E^{B_{R_{0}}^{X}\left( 0\right) }\right.
\right\} =M<\infty ,
\end{equation*}%
according to the conditions 1, 2 and boundedness of the norm $\left\Vert
Lx\right\Vert _{Y}$ (as $0<$ $\left\Vert Lx\right\Vert _{Y^{\ast
}}<M_{1}<\infty $, $x\in B_{R_{0}}^{X}\left( 0\right) $).

Consequently, there exists such constant $\lambda _{0}=\lambda _{0}\left(
M\right) $, that $0<\lambda _{0}<\infty $ and appriopreate $x_{\lambda _{0}}$
that $F_{\lambda _{0}}\left( x_{\lambda _{0}}\right) =0$.
\end{proof}

So, we provide the result on the spectrum of the operator $f\circ L$
relative to the operator $g$.

\begin{theorem}
\label{Th_6} Let functions $f$ and $g$ are homogeneous with the equal order
of nonlinearity, which is the continuous function $\varphi $ (i.e. for any $%
\tau \in R_{+}$ equalities $f\left( \tau \cdot y\right) =\varphi \left( \tau
\right) \cdot f\left( y\right) $, $g\left( \tau \cdot y\right) =\varphi
\left( \tau \right) \cdot g\left( y\right) $ hold). Assume all conditions
the above Proposition \ref{Pr_2} are fulfilled. Then operator $f\circ L$ has
a spectrum relative to operator $g$, which is a function of the spectrum of
operator $L$.
\end{theorem}

\begin{proof}
From Proposition \ref{Pr_2} follows the existence of such $\lambda \in 
\mathbb{%
\mathbb{R}
}_{+}$ and an element $x\in X$ that the equation $F_{\lambda }\left( 
\widetilde{x}_{\lambda }\right) \equiv f\left( L\widetilde{x}_{\lambda
}\right) -\lambda g\left( \widetilde{x}_{\lambda }\right) =0$ solvable. Then
using the well-known approach it is necessary to seek elements $\lambda
_{0}\in \mathbb{%
\mathbb{R}
}_{+}$ and $x_{0}\in X$, which satisfy the following equality 
\begin{equation}
\lambda =\inf \left\{ \frac{\left\langle f\left( Lx\right) ,Lx\right\rangle 
}{\left\langle g\left( x\right) ,Lx\right\rangle }\left\vert \ x\in X\right.
\right\} .  \label{4.5a}
\end{equation}

Due to the conditions of this theorem, it is enough to study the above
question only for $x\in S_{1}^{X}\left( 0\right) $. We can take into account
that $f$ is an $N-$function and the expression $\left\langle f\left(
Lx\right) ,Lx\right\rangle $ generates a functional $\Phi \left( Lx\right) $
according to the condition on $f$ \footnote{\begin{remark}
In the case when $L$ is the differential operator $\left\langle f\left(
Lx\right) ,Lx\right\rangle $ is a function of the norm $\left\Vert
Lx\right\Vert _{L_{\Phi }}$ of some Lebesgue or Orlicz space, where $\Phi $
be an $N-$function.
\end{remark}
} .

Whence follows that it is enough to seek the number $\lambda $ the following
way 
\begin{equation*}
\lambda =\inf \left\{ \frac{\left\Vert f\left( Lx\right) \right\Vert _{Z}}{%
\left\Vert g\left( x\right) \right\Vert _{Z}}\left\vert \ x\in
S_{1}^{X}\left( 0\right) \right. \right\} .
\end{equation*}

From above expression follows the existing number $\lambda >0$.

Thus one can state that there exists such number $\lambda $ that condition
3) of the Theorem \ref{Th_5} is fulfilled for $F_{\lambda }\left( \widetilde{%
x}_{\lambda }\right) $ with the mentioned number $\lambda $. Since according
to conditions of this theorem other conditions (i.e. conditions 1), 2), 4))
of Theorem \ref{Th_5} are fulfilled we obtain that all conditions of Theorem %
\ref{Th_5} are fulfilled for this case. Consequently, then using Theorem \ref%
{Th_5} we get the existence of an element $x_{0}\in S_{1}^{X}\left( 0\right) 
$ and the appropriate number $\lambda _{0}$ that is the element on which the
expression (\ref{4.5a}) attained the infimum $\lambda _{0}$.
\end{proof}

\begin{notation}
\label{N_3}In particular, if assume $x_{1}\in S_{1}^{X}\left( 0\right) $ is
the first eigenfunction and $\lambda _{1}$ first eigenvalue of the operator $%
L$ then we have 
\begin{equation*}
\lambda _{0}\leq \frac{\varphi \left( \lambda _{1}\right) \left\Vert f\left(
x_{1}\right) \right\Vert _{Z}}{\left\Vert g\left( x_{1}\right) \right\Vert
_{Z}}.
\end{equation*}
\end{notation}

Now we provide some examples of operators related to the above theorems.

\textbf{1.} Let $L:W^{m,p}\left( \Omega \right) \longrightarrow L_{p}\left(
\Omega \right) $ be a linear differential operator with the spectrum $%
P\left( L\right) \subset R_{+}$, the operator $f$ is the function $f\left(
\tau \right) =\left\vert \tau \right\vert ^{p-2}\tau $ and $g\equiv f$. So,
it needs to define the first eigenfunction and eigenvalue of the operator $%
f\left( L\circ \right) $ relative to operator $g\left( \circ \right) $. Then
using the expression (\ref{4.5a}) we get 
\begin{equation*}
\lambda _{f}=\inf \left\{ \frac{\left\langle f\left( Lu\right)
,Lu\right\rangle }{\left\langle g\left( u\right) ,Lu\right\rangle }%
\left\vert \ u\in W^{m,p}\left( \Omega \right) \right. \right\} =
\end{equation*}%
\begin{equation*}
=\inf \left\{ \frac{\left\Vert Lu\right\Vert _{L_{p}}^{p}}{\underset{\Omega }%
{\dint }\left( \left\vert u\right\vert ^{p-2}u\ Lu\right) dx}\left\vert \
u\in W^{m,p}\left( \Omega \right) \right. \right\} \geq
\end{equation*}%
\begin{equation*}
\geq \inf \left\{ \frac{\left\Vert Lu\right\Vert _{L_{p}}^{p-1}}{\left\Vert
u\right\Vert _{L_{p}}^{p-1}}\left\vert \ u\in W^{m,p}\left( \Omega \right)
\right. \right\} =
\end{equation*}%
\begin{equation*}
=\inf \left\{ \left( \frac{\left\Vert Lu\right\Vert _{L_{p}}}{\left\Vert
u\right\Vert _{L_{p}}}\right) ^{p-1}\left\vert \ u\in S_{1}^{W^{m,p}\left(
\Omega \right) }\left( 0\right) \right. \right\} .
\end{equation*}%
Whence we arrive that $\lambda _{f1}\geq \lambda _{L1}^{p-1}$, where $%
\lambda _{L1}$ is the first eigenvalue and the function $u_{1}\in
S_{1}^{W^{m,p}\left( \Omega \right) }\left( 0\right) $ is the first
eigenfunction of the operator $L$.

\textbf{2.} We will study the spectral property of the fully nonlinear
operator in the following two special cases 
\begin{equation}
-\left\vert \Delta u\right\vert ^{p-2}\Delta u=\lambda \left\vert \nabla
u\right\vert ^{\mu -2}u,\quad x\in \Omega ,\quad u\left\vert ~_{\partial
\Omega }\right. =0,  \label{4.2}
\end{equation}%
\begin{equation}
-\left\vert \Delta u\right\vert ^{p-2}\Delta u=\lambda \left\vert
u\right\vert ^{\nu }u,\quad x\in \Omega ,\quad u\left\vert ~_{\partial
\Omega }\right. =0,  \label{4.3}
\end{equation}%
i.e. we will seek of the spectrum of the operator $-\left\vert \Delta
u\right\vert ^{p-2}\Delta u$ separately relative to operators $\left\vert
\nabla u\right\vert ^{\mu -2}u$ and $\left\vert u\right\vert ^{\nu }u$,.

\textbf{2 (a).} Consider problem (\ref{4.2}). We will use the following
equality 
\begin{equation*}
\left\langle -\left\vert \Delta u\right\vert ^{p-2}\Delta u,-\Delta
u\right\rangle =\left\langle \lambda \left\vert \nabla u\right\vert ^{\mu
-2}u,-\Delta u\right\rangle
\end{equation*}%
then we get%
\begin{equation*}
\left\Vert \Delta u\right\Vert _{p}^{p}=\lambda \left( \mu -1\right)
^{-1}\left\Vert \nabla u\right\Vert _{\mu }^{\mu }.
\end{equation*}

Hence follows 
\begin{equation*}
\lambda _{1}\left( p,\mu \right) =\left( \mu -1\right) \inf \ \left\{ \frac{%
\left\Vert \Delta u\right\Vert _{p}^{p}}{\left\Vert \nabla u\right\Vert
_{\mu }^{\mu }}\left\vert \ u\in W^{2,p}\cap W_{0}^{1,p}\right. \right\} =\ 
\end{equation*}%
\begin{equation*}
\left( \mu -1\right) \inf \left\{ \frac{\left\Vert \Delta u\right\Vert _{p}}{%
\left\Vert \nabla u\right\Vert _{\mu }^{\mu /p}}\left\vert \ u\in
W^{2,p}\cap W_{0}^{1,p}\right. \right\}
\end{equation*}%
whence follows that to find $\lambda $, satisfying the assumed condition, we
must select the corresponded exponent $\mu $.

Consequently, we need to assume $\mu =p$, then we get 
\begin{equation}
\lambda _{1}\left( p,p\right) =\left( p-1\right) \inf \ \left\{ \frac{%
\left\Vert \Delta u\right\Vert _{p}}{\left\Vert \nabla u\right\Vert _{p}}%
\left\vert \ u\in W^{2,p}\left( \Omega \right) \cap W_{0}^{1,p}\left( \Omega
\right) \right. \right\} .  \label{4.4}
\end{equation}

As is well-known $\left\Vert \nabla u\right\Vert _{p}\leq c\left( p,\Omega
\right) \left\Vert \Delta u\right\Vert _{p}$ under the condition $%
u\left\vert \ _{\partial \Omega }\right. =0$, consequently, $\lambda
_{1}\left( p,p\right) \leq \left( p-1\right) c\left( p,\Omega \right) $.

\begin{proposition}
\label{Pr_5} Let $f_{0}:W^{2,p}\left( \Omega \right) \cap W_{0}^{1,p}\left(
\Omega \right) \longrightarrow L^{q}\left( \Omega \right) $ having
presentation $f_{0}\left( u\right) =-\left\vert \Delta u\right\vert
^{p-2}\Delta u$ and $f_{1}:W_{0}^{1,p}\left( \Omega \right) \longrightarrow
L^{q}\left( \Omega \right) $ having presentation $f_{1}\left( u\right)
=\left\vert \nabla u\right\vert ^{p-2}u$ . Then the operator $f_{0}$ has the
first eigenvalue relative to operator $f_{1}$, which is defined by (\ref{4.4}%
).
\end{proposition}

\textbf{2 (b).} Consider problem (\ref{4.3}) for $\nu =p-2$ then we have%
\begin{equation*}
\left\langle -\left\vert \Delta u\right\vert ^{p-2}\Delta u,-\Delta
u\right\rangle =\left\langle \lambda \left\vert u\right\vert ^{p-2}u,-\Delta
u\right\rangle \Longrightarrow
\end{equation*}%
\begin{equation*}
\left\Vert \Delta u\right\Vert _{p}^{p}=\lambda \frac{4\left( p-1\right) }{%
p^{2}}\left\Vert \nabla \left( \left\vert u\right\vert ^{\frac{p-2}{2}%
}u\right) \right\Vert _{2}^{2}
\end{equation*}%
or 
\begin{equation*}
\left\Vert \Delta u\right\Vert _{p}^{p}=\lambda \left( p-1\right) \left\Vert
\left( \left\vert u\right\vert ^{p-2}\left\vert \nabla u\right\vert
^{2}\right) \right\Vert _{1}.
\end{equation*}%
Thus we get 
\begin{equation*}
\widetilde{\lambda }_{1}\left( p\right) =\frac{1}{p-1}\inf \left\{ \ \frac{%
\left\Vert \Delta u\right\Vert _{p}^{p}}{\left\Vert \left\vert u\right\vert
^{\frac{p-2}{2}}\left\vert \nabla u\right\vert \right\Vert _{2}^{2}}%
\left\vert \ u\in W^{2,p}\cap W_{0}^{1,p}\right. \right\} =
\end{equation*}%
\begin{equation}
=\frac{p}{2\left( p-1\right) }\inf \left\{ \ \frac{\left\Vert \left\vert
\Delta u\right\vert ^{\frac{p}{2}}\right\Vert _{2}}{\left\Vert \nabla \left(
\left\vert u\right\vert ^{\frac{p-2}{2}}\left\vert u\right\vert \right)
\right\Vert _{2}}\left\vert \ u\in W^{2,p}\cap W_{0}^{1,p}\right. \right\}
\label{4.5}
\end{equation}%
Hence we obtain 
\begin{equation}
\widetilde{\lambda }_{1}\left( p\right) \geq \frac{1}{p-1}\frac{\left\Vert
\Delta u\right\Vert _{p}^{p}}{\left\Vert u\right\Vert _{p}^{p-2}\left\Vert
\nabla u\right\Vert _{p}^{2}}  \label{4.6}
\end{equation}%
according to the following inequality 
\begin{equation*}
\left\Vert \left( \left\vert u\right\vert ^{p-2}\left\vert \nabla
u\right\vert ^{2}\right) \right\Vert _{1}\leq \left\Vert u\right\Vert
_{p}^{p-2}\left\Vert \nabla u\right\Vert _{p}^{2}.
\end{equation*}

So, we arrive at the result

\begin{proposition}
\label{Pr_3}Let $f_{0}:W^{2,p}\left( \Omega \right) \cap W_{0}^{1,p}\left(
\Omega \right) \longrightarrow L^{q}\left( \Omega \right) $ having
presentation $f_{0}\left( u\right) =-\left\vert \Delta u\right\vert
^{p-2}\Delta u$ and $f_{1}:L^{p}\left( \Omega \right) \longrightarrow
L^{q}\left( \Omega \right) $ having presentation $f_{1}\left( u\right)
=\left\vert u\right\vert ^{p-2}u$ . Then the operator $f_{0}$ has the first
eigenvalue relative to operator $f_{1}$, which is defined by (\ref{4.5}) and
satisfies the inequation (\ref{4.6}).
\end{proposition}

\begin{remark}
\label{R_3} In the previous case, would be to use the following equality 
\begin{equation*}
-\underset{\Omega }{\int }\left\vert \Delta u\right\vert ^{p-2}\Delta
udx=\lambda \underset{\Omega }{\int }\left\vert u\right\vert
^{p-2}udx\Longrightarrow
\end{equation*}%
whence follows%
\begin{equation*}
\left\Vert \Delta u\right\Vert _{p-1}=\lambda \left\Vert u\right\Vert
_{p-1}\Longrightarrow
\end{equation*}%
then we will get 
\begin{equation*}
\lambda =\inf \left\{ \frac{\left\Vert \Delta u\right\Vert _{p-1}}{%
\left\Vert u\right\Vert _{p-1}}\left\vert \ u\in W^{2,p}\cap
W_{0}^{1,p}\right. \right\}
\end{equation*}%
since in the conditions of this section the operator $-\Delta $ is positive.
\end{remark}

\subsubsection{Conclusions}

\textit{To seek the eigenvalues of the nonlinear continuous operator in the
Banach space is necessary to choose the other operator in such a way that
the order of nonlinearity will be identical whit the order of nonlinearity
of the examined operator. If one uses the proposed approach then is possible
to find the other eigenvalues of this operator. The importance of the
knowledge of the eigenvalues of the operators in the study of the
bifurcation of solutions of nonlinear equations shows many articles
dedicated to studying the bifurcation of the solutions of nonlinear
equations (see, e.g. Section 3 and also articles \cite{12, 16, 33}, etc.). }

\bigskip 

\end{document}